\documentclass[AIF,Unicode,manuscript,french]{cedram}

\usepackage{mathrsfs}

\usepackage{hyperref} 
\usepackage{amssymb, amsmath, mathabx}
\usepackage[shortlabels]{enumitem}
\usepackage{tikz-cd}

\newcommand{\eg}{e.\@g.\@}

\newcommand{\ie}{i.\@e.\@}

\newcommand{\abs}[1]{\mathopen|#1\mathclose|}
\newcommand{\adhp}[2]{\overline{ \{#1\}}^{#2_{\mathrm {Zar}}}}

\newcommand{\an}{^{\mathrm{an}}}
\newcommand{\al}{^{\mathrm{al}}}

\newcommand{\dimc}{\dim_{\mathrm c}}

\newcommand{\grot}{_{\mathrm G}}

\newcommand{\hr}[1]{\mathscr H(#1)}
\newcommand{\hrt}[1]{\widetilde{\mathscr H(#1)}}

\newcommand{\pl}[2]{\mathsf P(#1/#2)}
\newcommand{\ql}[2]{\mathsf Q(#1/#2)}
\newcommand{\qln}[3]{\mathsf Q(#1/#2)_{\geq #3}}
\newcommand{\qlns}[3]{\mathsf Q(#1/#2)_{\geq #3}^
{\mathrm {sat}}}

\def\spec{\mathop{\mathrm{Spec}}}

\newcommand{\supp}[1]{\mathrm{Supp}(\mathscr #1)}

\DeclareMathOperator{\colim}{colim}

\newcommand{\A}{\mathbf A}

\newcommand{\Q}{\mathbf Q}
\newcommand{\R}{\mathbf R}
\newcommand{\Z}{\mathbf Z}

\renewcommand{\phi}{\varphi}
\renewcommand{\epsilon}{\varepsilon}
\renewcommand{\leq}{\leqslant}
\renewcommand{\geq}{\geqslant}

\SetEnumerateShortLabel{1}{\textnormal{(\arabic{enumi})}}
\SetEnumerateShortLabel{s}{\textnormal{(\arabic{enumi}$^\ast$)}}
\SetEnumerateShortLabel{i}{\textnormal{(\roman{enumi})}}
\SetEnumerateShortLabel{a}{\textnormal{(\alph{enumi})}}
\SetEnumerateShortLabel{A}{\textnormal{(\Alph{enumi})}}
\SetEnumerateShortLabel{2}{\textnormal{(\arabic{enumii})}}
\SetEnumerateShortLabel{j}{\textnormal{(\roman{enumii})}}
\SetEnumerateShortLabel{b}{\textnormal{(\arabic{enumi}\alph{enumii})}}
\SetEnumerateShortLabel{c}{\textnormal{(\alph{enumii})}}
\SetEnumerateShortLabel{B}{\textnormal{(\Alph{enumii})}}

\title{Utilisation de l'aplatissement en géométrie de Berkovich}

\alttitle{Application of  flattening techniques in Berkovich geometry}

\author{\firstname{Antoine} \lastname{Ducros}}
\address{Sorbonne Université, Université Paris-Diderot, CNRS, Institut de Mathématiques de Jussieu-Paris
Rive Gauche, IMJ-PRG, 
F-75005, Paris, France
}
\email{antoine.ducros\at imj-prg.fr}
\thanks{Ce travail a démarré et été réalisé en grande partie à l'institut Weizmann (Rehovot, Israël) où je
suis resté du 15 septembre au 13 décembre 2022. Merci à Vladimir Berkovich de m'y avoir invité, et merci à cet établissement de m'avoir accueilli et d'avoir financé mon séjour
par une {\sc Emmma and Michael Jacob} \textit{professorship}. J'ai également bénéficié durant sa rédaction du soutien du projet ANR-24-CE40-6184 (AdAnAr).}
\urladdr{http://www.imj-prg.fr/$\sim$antoine.ducros}

\begin{abstract}Dans cet article, nous mettons en œuvre les techniques d'aplatissement développées dans un précédent travail pour «enjoliver» 
un morphisme entre espaces analytiques compacts, pour décrire la structure de son image et obtenir ainsi un substitut non archimédien au théorème de Chevalley, et enfin pour montrer que la platitude dans le monde des espaces de Berkovich revient à la platitude naïve à condition de considérer les anneaux
locaux pour la G-topologie.
\end{abstract}

\begin{altabstract}In this article, we carry out the flattening techniques developped in a former work in order to "embellish" a map between 
compact analytic spaces, to describe the structure of its image, getting this way a substitute for Chevalley's theorem in the non-archimedean setting, and
finally to show that flatness in the world of Berkovich spaces amounts to naive flatness provided one works with local rings for the G-topology.\end{altabstract}

\altkeywords{Berkovich spaces, Images of analytic maps,flatness}

\begin{document}

\maketitle
\tableofcontents

\setcounter{section}{-1}

\section{Introduction}

Nous proposons dans le présent travail plusieurs applications
des techniques d'aplatissement non archimédiennes mises au point
par l'auteur dans \cite{ducros2021a}. Précisons d'emblée que si 
ces techniques sont directement inspirées de
celles de l'article fondateur \cite{raynaud-g1971}
de Raynaud et Gruson, elles donnent lieu à des énoncés 
nettement plus compliqués que leurs originaux schématiques. En effet, 
nous ne réussissons à aplatir un faisceau cohérent que par transformée
stricte après
changement de base le long
d'une succession d'éclatements \textit{et de morphismes quasi-étales}
(au paragraphe final 7.12 de \cite{ducros2021a}, nous expliquons
pourquoi il n'est probablement pas possible d'éviter ces derniers).

Si $X$ est un espace $k$-analytique compact, nous aurons donc très souvent à considérer
un couple $(Z,S)$ formé
d'une part d'un espace $k$-analytique compact $Z$ muni d'une flèche 
$Z\to X$ qui est une tour d'éclatements et de morphismes quasi-étales
à sources compactes,
et d'autre part d'un diviseur de Cartier $S$ de $Z$ contenant les images réciproques 
des centres de tous les éclatements intermédiaires de la tour (si bien que $Z\setminus S\to X$ est quasi-étale). 
Comme ces couples nous permettront d'«enjoliver» les morphismes, nous avons choisi de les appeler
\textit{enjoliveurs} de $X$ ; ils sont introduits et brièvement étudiés à la section \ref{s-enjoliveurs}, qui se conclut
par une proposition assurant qu'un enjoliveur d'un sous-espace analytique fermé de $X$ peut toujours être relevé
en un enjoliveur de $X$ (proposition \ref{prop-enjoliveur-ferme}).

Après avoir consacré la section \ref{s-aplatissement}
à la présentation d'une version un peu améliorée (théorème \ref{theo-aplatissement-bis})
du théorème principal d'aplatissement de \cite{ducros2021a}, nous exploitons
cette version à la section \ref{s-comment} pour enjoliver 
un morphisme.

\begin{theo}[Version simplifiée du
théorème \ref{theo-embellissement}]\label{theo-embellissement-intro}
Soit $Y\to X$
un morphisme entre espaces $k$-analytiques compacts
et réduits.  
On suppose que $Y$ est non vide et 
équidimensionnel.
On pose $n=\dim Y$  
et l'on suppose que la dimension relative 
de $Y$ au-dessus de $X$ a même valeur générique (ou encore minimale) $d$
sur chaque composante
irréductible de $Y$.

Il existe alors un enjoliveur $(Z,S)$ de $X$
et un domaine analytique compact $V$ de $Y\times_X Z$ tels que les propriétés suivantes soient satisfaites, 
en notant $V'$ la «transformée stricte» de $V$, c'est-à-dire l'adhérence réduite de 
$V\setminus (V\times_Z S)$
dans $V$ :
\begin{itemize}[label=$\diamond$] 
\item l'image de $S$ sur $X$ est contenue dans $f(Y)$ et de dimension $<n-d$ 
(la dimension d'une partie quelconque
d'un espace analytique est définie à la section \ref{rappels}) ; 
\item $V'\to Y$ est surjectif ; 
\item $V'\to Z$ se factorise par un morphisme surjectif et plat sur un sous-espace
analytique fermé et réduit de $Z$, transverse à $S$ et purement de dimension $n-d$. 
\end{itemize}
\end{theo}

Nous nous intéressons ensuite à la structure de $f(Y)$
où $f\colon Y\to X$ est un morphisme entre espaces $k$-analytiques compacts. 
Précisons qu'à notre connaissance 
la structure de $f(Y)$ n'était jusqu'ici bien comprise que lorsque $f$ est plat (c'est alors un domaine analytique compact de $X$, 
par le théorème 9.2.1 de \cite{ducros2018}, dû à Raynaud dans le cas strict, \textit{cf.} \cite[cor. 5.11]{frg2}) ou propre
(c'est alors un fermé de Zariski de $X$ ; on le déduit du théorème de cohérence des images directes établi par Kiehl
dans \cite{kiehl1967b}, \textit{cf.} \cite[1.3.23]{ducros2018} pour les détails). Mentionnons toutefois que des
stratifications
raisonnables d'images de morphismes analytiques ont été construites par Lipshitz 
dans un tout autre langage et avec des méthodes entièrement différentes (son cadre de travail est la théorie des modèles des corps 
valués algébriquement clos dans un langage enrichi par une classe de fonctions analytiques «raisonnables»), voir par exemple le théorème 5.2
de \cite{lipshitz1993} ; nous ignorons les liens précis éventuels entre ses théorèmes de structure et les nôtres.

Nous considérons tout d'abord le
cas particulier d'un morphisme quasi-étale
$f\colon Y\to X$ entre espaces $k$-analytiques. Pour des raisons techniques, nous 
faisons simplement une hypothèse 
de compacité \textit{relative} : nous demandons que $f$ soit compact, c'est-à-dire que
$f^{-1}(K)$ soit compact pour tout compact $K$ de $X$. 

\begin{theo}[Version simplifiée du 
théorème \ref{theo-casqe}]\label{theo-casqe-intro}
Soit
$f\colon Y\to X$ un morphisme compact 
entre espaces $k$-analytiques.
Soit $m$ un entier et soit $Z$
un fermé de Zariski de $Y$ purement de dimension $m$. 
Pour tout $d$, désignons par $X_d$ l'ensemble des points de $X$
dont la fibre géométrique sous $f$ est de cardinal au moins $d$.
Le sous-ensemble $X_d$ de
$X$ en est un domaine analytique fermé, et l'intersection
$f(Z)\cap (X_d\setminus X_{d+1})$ est pour tout $d$ 
un fermé de Zariski de $X_d\setminus X_{d+1}$.
\end{theo}

Ce théorème se démontre par descente quasi-étale.
Il fournit en particulier
une partition de $f(Z)$ en sous-ensembles
localement fermés, chacun d'eux étant un fermé de Zariski purement de dimension $m$
d'un domaine analytique de $X$. 
Remarquons que si le cardinal des
fibres géométriques de $f$ est majoré (ce sera toujours le cas dans les situations que nous rencontrerons), 
cette partition est finie car $X_d$ est alors vide pour $d$ assez grand.

Nous utilisons ensuite ce théorème pour
donner une description de l'image d'un morphisme général entre 
espaces $k$-analytiques compacts, que nous voyons comme un 
théorème de Chevalley non archimédien. 
Dans ce but, nous introduisons tout d'abord la définition suivante.

\begin{defi}[Version simplifiée de la définition 
\ref{def-cellule}]
Soit $X$ un espace $k$-analytique et soit $m$ un entier. On définit
la notion de \emph{$m$-cellule} de $X$ récursivement sur $m$
de la façon suivante :

\begin{itemize}[label=$\diamond$]
\item Une $0$-cellule de $X$ est un ensemble
fini et non vide de points rigides de $X$. 
\item Si $m>0$, une $m$-cellule de $X$
est une partie localement fermée $C$
non vide de $X$ telle que les propriétés suivantes soient satisfaites, en notant
$\partial C$ le fermé $\overline C\setminus C$ : 
\begin{enumerate}[a]
\item il existe 
un morphisme quasi-étale compact $g\colon X'\to X\setminus \partial C$, dont le cardinal
des fibres géométriques est majoré, et un fermé de Zariski $Z\subset X'$
non vide
et purement de dimension $m$ tel que $g(Z)=C$ 
(ainsi, $C$ admet une partition finie de la forme
fournie par le théorème \ref{theo-casqe-intro}) ; 

\item $\partial C$ s'écrit comme une union finie $\bigcup \overline{C_j}$ où
chaque $C_j$ est une $m_j$-cellule pour un certain entier $m_j<m$.
\end{enumerate}

\end{itemize}

\end{defi}

\begin{theo}[Version simplifiée
du théorème \ref{theo-chevalley-na}, avatar non archimédien
du théorème de Chevalley]\label{theo-chevalley-na-intro}

Soit $f\colon Y\to X$ un morphisme entre espaces
$k$-analytiques compacts. On suppose que $Y$
est non vide et équidimensionnel.
On pose $n=\dim Y$  
et l'on suppose que la dimension relative 
de $Y$ au-dessus de $X$ a même valeur générique (ou encore minimale) $d$
sur chaque composante
irréductible de $Y$. 
Il existe alors une $(n-d)$-cellule $C$
de $X$ telle que $f(Y)=\overline C$
et 
telle que $f(y)\in C$ pour tout point d'Abhyankar de rang $n$ de $Y$. 
\end{theo}

La preuve consiste à enjoliver $f$ pour se ramener au cas où il admet une factorisation 
$Y\hookrightarrow X_0
\to X$ où $X_0$ est compact, où $Y\hookrightarrow X_0$ 
est une immersion fermée et où $g\colon X_0\to X$ est quasi-étale en dehors d'un
diviseur de Cartier $S$ transverse à $Y$ dont l'image est contenue dans $f(Y)$
et de dimension $<n-d$.
Si l'on pose $X'=g^{-1}(g(S))$ la cellule $C$ du théorème est alors $g(Y\cap X')$ 
et son bord $\partial C$ est égal à $g(S)$.
La propriété (a) est évidente une fois établi que les fibres
géométriques du morphisme quasi-étale et compact $X'\to X\setminus g(S)$ sont de cardinal borné 
(ce qui n'est pas difficile), et on
obtient (b) en faisant une récurrence
sur la dimension de l'image de $f$.

\paragraph{Applications à la platitude}
Nous terminons ce travail en établissant de nouveaux critères de platitude en géométrie de Berkovich, que nous espérons
plus maniables que la définition
originelle. Commençons par quelques
rappels à propos de cette dernière.

Soit $Y\to X$  un morphisme entre bons espaces analytiques sur un même
corps de base $k$. 
Si $\mathscr F$ est un faisceau cohérent sur $Y$ et si $y$ est un point de $Y$ d'image $x$ sur $X$, on dit que $\mathscr F$ est \textit{naïvement}
plat sur $X$ en $y$ si $\mathscr F_y$ est un $\mathscr O_{X,x}$-module plat.
L'adverbe «naïvement» est utilisé parce que cette propriété manque de robustesse : elle n'est en
effet pas stable en général par changement de base bon, un contre-exemple dû à Temkin étant étudié en détail à la section 
4.4 de \cite{ducros2018}.
Pour cette raison, on dit que $\mathscr F$ est $X$-plat en $y$ s'il est naïvement plat en $y$
\textit{et si cette propriété perdure après tout changement de base bon}.
Il y a toutefois un certain nombre de cas
dans lesquels on sait que la platitude naïve de $\mathscr F$ en $y$ est suffisante, c'est-à-dire qu'elle entraîne 
sa platitude en $y$ au sens précédent :  celui où
$y$ appartient à $\mathrm{Int}(Y/X)$\cite[ th. 8.3.4]{ducros2018}, et celui où $X$ est réduit et où $x$ appartient à l'ensemble
$\mathsf A(X)$ des points de $X$ qui sont d'Abhyankar et de rang maximal (sous cette hypothèse
l'anneau local $\mathscr O_{X,x}$ est un corps, si bien que la
platitude naïve de $\mathscr F$ en $y$ est dès lors automatique, et le théorème 10.3.7 de \cite{ducros2018} assure que $\mathscr F$ est plat
en $y$).

Le contre-exemple de Temkin  étudié 
à la section 4.4 de \cite{ducros2018}
est construit en exhibant un point $x$ du plan affine $\A^{2,\mathrm{an}}_k$
tel que $\adhp xU$ soit de dimension $2$
pour tout \textit{voisinage} analytique de $x$ dans $\A^{2,\mathrm{an}}_k$, et un bon domaine
analytique $V$ de $\A^{2,\mathrm{an}}_k$ contenant $x$ et vérifiant l'égalité
$\dim \adhp xV=1$. L'anneau local $\mathscr O_{X,x}$ est alors un corps, si bien que la platitude 
naïve au-dessus de $x$ est automatique.  Si $Y$ désigne l'adhérence réduite de $x$ dans $V$, 
l'immersion $Y\hookrightarrow X$ est donc naïvement plate en $x$, mais
$Y\times_X V\hookrightarrow V$, qui n'est autre que l'immersion
fermée de la courbe $Y$ dans la surface $V$, ne l'est pas. 
Au cœur de ce contre-exemple 
figure ainsi la chute de la dimension de l'adhérence de Zariski d'un point lorsqu'on
se restreint à un domaine analytique bien choisi. Nous démontrons
ici que ce phénomène géométrique est en fait \textit{la seule
obstruction} à la détection naïve de la platitude.

\begin{theo}[Version
simplifiée du théorème \ref{theo-plat-naif}]\label{theo-plat-naif-intro}
Soit $Y\to X$ un morphisme
entre bons espaces $k$-analytiques et soit $y$ un point de $Y$ dont on note $x$
l'image sur $X$ ; supposons que $\dim \adhp xV=\dim \adhp xX$ pour tout domaine analytique $V$ de $X$ contenant $x$. 
Un faisceau cohérent sur $Y$ est
alors $X$-plat en $y$ si et seulement s'il est naïvement $X$-plat en $y$.
\end{theo}

On démontre ce théorème
en se ramenant, grâce aux critères locaux de platitude
standard en algèbre commutative, au cas où
l'espace $X$ est réduit et où
l'on a $\dim \adhp xX=\dim \adhp xV=\dim_x X$
pour tout domaine analytique $V$ de $X$
contenant $x$, ce que nous résumerons dans la discussion qui suit en disant que $x$ est \textit{très générique}
(mentionnons à titre d'exemple que si $x\in \mathsf A(X)$ il est très générique, 
et que si $X$ est une courbe et si $x$ n'est pas rigide, il est très générique). 
Nous suivons alors peu ou prou la preuve du théorème 10.3.7  de \cite{ducros2018} (qui est essentiellement 
le théorème \ref{theo-plat-naif} dans le cas particulier mentionné plus haut où $X$ est réduit et où $x\in \mathsf A(X)$), mais avec une modification importante : 
nous remplaçons la transitivité de l'appartenance à $\mathsf A(\cdot)$ par la transitivité du caractère très générique
(théorème \ref{tres-generique-robuste}), dont la démonstration utilise elle-même de manière cruciale notre description de l'image
d'un fermé de Zariski par un morphisme quasi-étale compact (théorème \ref{theo-casqe-intro}). 

Une conséquence du théorème \ref{theo-plat-naif-intro}
est que la platitude en un point $x$ d'un
espace analytique $X$ peut en fait \textit{toujours} être détectée naïvement, dès
lors qu'on travaille
avec l'anneau local G-topologique $\mathscr O_{X\grot, x}$ qui est la colimite
des $\mathscr O_X(V)$ où $V$ parcourt l'ensemble des domaines analytiques de $X$
contenant $x$ (et pas uniquement celui des voisinages analytiques de $X$, comme c'est le cas
pour l'anneau local usuel $\mathscr O_{X,x}$).

\begin{theo}[Version simplifiée du  théorème \ref{plat-naif-gtop}]\label{plat-naif-gtop-intro}
Soit $Y\to X$ un morphisme entre espaces $k$-analytiques, soit $y$
un point de $Y$ et soit $x$ son image sur $X$. Soit $\mathscr F$ un faisceau cohérent
sur $Y$. Le faisceau $\mathscr F$ est $X$-plat en $y$ si et seulement si
le $\mathscr O_{X\grot,x}$-module $\mathscr F_{Y\grot, y}$ est plat.
\end{theo}

Nous nous assurons par ailleurs que 
ces anneaux locaux G-topologiques
ont de très bonnes propriétés algébriques puisqu'ils sont henséliens et excellents : 
c'est l'objet de la proposition \ref{prop-otheta} ; sa preuve
utilise de manière essentielle le fait que ces propriétés
sont satisfaites par les anneaux locaux usuels des bons espaces analytiques.

\paragraph{Généralisation : les filtres affinoïdes}
Dans ce que nous venons de décrire, un rôle majeur est joué par l'ensemble de tous les domaines analytiques contenant un point donné
$x$ d'un espace analytique $X$, ensemble qu'on peut de façon évidente 
remplacer par n'importe lequel de ses sous-ensembles cofinaux, comme par exemple
celui de tous les domaines affinoïdes de $X$ contenant $x$.
Mais une bonne partie de nos raisonnements s'appliquent
également à d'autres ensembles plus restreints de domaines affinoïdes de $X$ contenant $x$, qui peuvent avoir
leur intérêt : lorsque la valeur absolue de $k$ n'est pas triviale et que $X$ est strict, c'est par exemple le cas 
de l'ensemble des domaines strictement affinoïdes de $X$ contenant $x$, ou de l'ensemble des domaines strictement
affinoïdes $V$ de $X$ tels que $V^{\mathrm{ad}}$ contienne une spécialisation
donnée $x^+$ de $x$ dans $X^{\mathrm{ad}}$, où le suffixe «ad» en exposant fait référence à l'espace adique associé. 

Pour couvrir ces différents cas de manière aussi uniforme que possible, nous travaillons avec la notion
de  \textit{filtre affinoïde} sur un espace pointé $(X,x)$, introduite à la définition \ref{def-filtreaff}. 
Ainsi,
la proposition \ref{prop-otheta} que nous avons mentionnée plus haut assure
par exemple en fait que 
l'anneau local associé à \textit{n'importe quel} filtre affinoïde est hensélien et excellent, 
et cela s'applique en particulier aux anneaux locaux d'un espace adique de type fini sur un corps
(corollaire \ref{coro-huber-excellent}).

\paragraph{Utilisations du présent travail}
Notre procédé pour enjoliver les morphismes 
(théorème \ref{theo-embellissement-intro}) et notre théorème de Chevalley non archimédien
(théorème \ref{theo-chevalley-na-intro}) jouent un rôle essentiel
dans l'article \cite{ducros-t2024} où nous étudions avec Amaury Thuillier les images directes de squelettes.  
Les anneaux locaux G-topologiques $\mathscr O_{X\grot, x}$ et le critère de platitude afférent (théorème 
\ref{plat-naif-gtop-intro}) interviennent dans la version récente de notre travail commun \cite{chambertloir-d2025}
avec Antoine Chambert-Loir 
sur les formes différentielles et courants réels en géométrie non archimédienne. 

Nous avons par ailleurs comme projet, à beaucoup plus long terme, 
d'utiliser notre théorème de Chevalley non archimédien pour dégager une notion
raisonnable de faisceaux constructibles sur un espace analytique. 

\paragraph{Remerciements}
Je sais gré au rapporteur anonyme de sa relecture extrêmement minutieuse du manuscrit et des corrections très pertinentes
qu'il a suggérées.

\section{Rappels, notations, références}\label{rappels}

Dans tout
ce texte nous travaillerons avec
les espaces analytiques ultramétriques
au sens de Berkovich, 
tels que définis
dans \cite{berkovich1993} (les espaces
qu'il avait auparavant introduits dans
\cite{berkovich1990} sont ceux qu'on
appelle maintenant les \textit{bons}
espaces, c'est-à-dire ceux qui ont une 
base de voisinages affinoïdes). 

\paragraph{Généralités}
Si $x$ est un point d'un espace
$k$-analytique $X$, son corps résiduel
complété sera noté $\hr x$. Si
$\phi \colon Y\to X$ est un morphisme
d'espaces $k$-analytiques, sa fibre
en un point $x$ sera notée
$\phi^{-1}(x)$ ou $Y_x$ ; c'est un 
espace $\hr x$-analytique.

Un espace $k$-analytique $X$ est muni
d'une topologie au sens classique et 
d'une topologie ensembliste, 
la G-topologie, qui la raffine. Le site
correspondant $X\grot$ est muni
d'un faisceau de $k$-algèbres que 
Berkovich note $\mathscr O_{X\grot}$ 
mais que nous noterons simplement
$\mathscr O_X$. Il est cohérent
(\cite[Lemme 0.1]{ducros2009} ; voir
aussi
la note de bas de page
de \cite[1.3.1]{ducros2018}
pour la 
rectification d'une
erreur dans la preuve signalée par J. Poineau). 
Nous appellerons \textit{faisceau
cohérent} sur $X$ tout $\mathscr O_X$-module cohérent ; il s'agit donc d'un
faisceau sur le site $X\grot$. 
Si $\mathscr F$ est un faisceau cohérent
sur $X$ et si $\phi\colon Y\to X$
est un morphisme d'espaces
$k$-analytiques, nous noterons
$\mathscr F_Y$ le faisceau 
cohérent $\phi^*\mathscr F$ 
sur $Y$. 

L'espace $X$ possède aussi une topologie
naturelle plus grossière que sa topologie
usuelle,
la topologie \textit{de Zariski}, 
dont les fermés sont les lieux des
zéros des faisceaux cohérents d'idéaux, 
\textit{cf.} \cite{ducros2009}. 
Le caractère ouvert (ou fermé) pour la
topologie de Zariski est G-local, 
\textit{cf.} \cite[prop. 4.2]{ducros2009} ; mais on prendra
garde que si $V$ est un
domaine analytique de $X$, la
topologie de Zariski de $V$ est en 
général plus fine que la topologie
induite par la topologie
de Zariski de $X$. 

Lorsque nous dirons qu'une propriété vaut \textit{génériquement} 
sur un espace $X$, cela signifiera que son lieu de validité
contient un ouvert \textit{de Zariski} dense de $X$.

\paragraph{Espaces $\Gamma$-stricts}
Soit $k$ un corps
ultramétrique complet
et soit
$\Gamma$ un sous-groupe
de $\R_+^\times$ tel
que $\Gamma \cdot \abs{k^\times}$
soit non trivial ; autrement dit, 
$\Gamma$ est non trivial si la valeur
absolue de $k$ est triviale. 
Nous utiliserons abondamment 
la notion d'espace $k$-analytique
\textit{$\Gamma$-strict} 
introduite au chapitre
3 de \cite{ducros2018}. 
Informellement, un espace
$k$-analytique est $\Gamma$-strict 
s'il peut être défini en ne faisant
intervenir que des paramètres réels
appartenant à $\Gamma$ ;
ainsi tout espace $k$-analytique
est $\R_+^\times$-strict,
et si $k$ n'est
pas trivialement valué,
un espace $k$-analytique
est $\{1\}$-strict si et seulement
s'il est strict. 
(La non-trivialité
de $\Gamma\cdot
\abs{k^\times}$
sert à
garantir que tout point
d'un espace affinoïde $\Gamma$-strict
a une base de voisinages affinoïdes
et $\Gamma$-stricts.)

Si $X$ est un espace $k$-analytique $\Gamma$-strict et si
$Y$ est un sous-espace analytique fermé de $X$, tout domaine analytique
$\Gamma$-strict de $Y$ est G-recouvert par des domaines de la forme
$V\cap Y$ où $V$ est un domaine analytique $\Gamma$-strict de $X$: c'est 
une conséquence immédiate de la version $\Gamma$-stricte du théorème de 
Gerritzen-Grauert, voir le paragraphe 0.20.7
de \cite{ducros2012b}. 

\paragraph{Déclinaison analytique de certaines
notions schématiques }
Nous nous servirons librement des 
analogues en théorie de Berkovich d'un certain nombre de notions 
de base de théorie des schémas : la théorie de la dimension, pour laquelle nous
renvoyons par exemple
à l'article \cite{ducros2007} ; la théorie
des composantes irréductibles
d'un espace analytique
développée
à la section 4
de \cite{ducros2009} ; la théorie de la 
connexité et de l'irréductibilité \textit{géométriques} \cite[section 7]{ducros2009}, 
section 7) ; et les propriétés usuelles de l'algèbre commutative 
(caractère réduit, régulier, etc.) ainsi que leurs variantes géométriques 
dans le contexte analytique 
(elles sont largement étudiées dans
\cite{ducros2009}, 
mais l'on pourra aussi se reporter aussi au chapitre 2 de \cite{ducros2018}). 

Si 
$\phi \colon Y\to X$ est un morphisme d'espaces $k$-analytiques et
$y$ un point de $Y$ d'image
$x$ sur $X$,
nous noterons $\dim_y \phi$
et appellerons dimension 
de $\phi$ en $y$,
ou dimension relative
de $Y$ sur $X$ en $y$,
la dimension en $y$ de la fibre $Y_x$.
La fonction $y\mapsto \dim_y \phi$
est semi-continue supérieurement
pour la topologie de Zariski 
sur $Y$ (\cite{ducros2007},
théorème 4.9). 

Soit $X$
un espace $k$-analytique et soit $x$ un point de $X$. Nous noterons $d_k(x)$ la somme du
degré de transcendance résiduel de
$\hr x$ sur $k$ et du rang rationnel de
$\abs{\hr x^\times}/\abs{k^\times}$
(on peut aussi décrire $d_k(x)$
de façon un peu plus ramassée comme
le degré de transcendance \textit{résiduel gradué}
de $\hr x$
au sens de \cite{temkin2004}, voir aussi l'appendice A de \cite{ducros2018}). 
L'intérêt technique majeur de cet invariant
provient de l'égalité
\[\dim X=\sup_{x\in X}d_k(x)\]
(avec la convention $\dim \emptyset=-\infty$). En particulier 
on a pour tout $x\in X$ l'inégalité $d_k(x)\leq \dim_x X$, et nous noterons
$\mathsf A(X)$ l'ensemble des points de $X$ pour lesquels on a égalité. Il est immédiat que $x\in
\mathsf A(X)$ si et seulement si
l'adhérence de Zariski $\adhp xX$ est une composante irréductible de $X$ de dimension $d_k(x)$
et que dans ce cas, 
$\adhp xV$ est une composante irréductible de $V$ de dimension 
$d_k(x)$ pour tout domaine
analytique $V$ de $X$ contenant $x$.

Ce qui précède permet 
d'étendre la définition de la dimension
(globale et locale) 
aux parties quelconques
de $X$ : si $E$ est un sous-ensemble
de $X$, on définit sa dimension 
comme le supremum
des $d_k(x)$ pour $x\in E$, et si $x\in E$, on définit $\dim_x E$
comme le minimum des $\dim U$ où $U$ parcourt l'ensemble
des voisinages ouverts de $x$ dans $E$.

Précisons
enfin que même si cela n'apparaît
pas dans les notations, 
la dimension en géométrie 
de Berkovich est toujours implicitement
\textit{relative au corps de base $k$}
(et est donc \textit{stricto sensu}
une
dimension \textit{$k$-analytique},
bien que ce qualificatif
soit la plupart
du temps omis). Par
exemple si $r\notin \abs{k^\times}$
et si l'on note $k_r$ le corps
$k$-affinoïde $k\{r^{-1}T,rT^{-1}\}$, la 
dimension $k$-analytique
de $\mathscr M(k_r)$ est
égale à $1$ alors 
que sa dimension 
$k_r$-analytique est égale à $0$. 
Mais en pratique cela ne prêtera
pas à confusion, le corps de base étant
toujours clairement
indiqué par le contexte. 
Indiquons simplement pour
qu'il n'y ait pas d'ambiguïté que
si une partie $E$ de $X$
est explicitement définie comme
contenue 
dans la fibre en un point $t$ d'un
morphisme d'espaces
$k$-analytiques $X\to T$, sa
dimension désignera sauf mention expresse
du contraire
sa dimension
$\hr t$-analytique, c'est-à-dire le 
supremum des
$d_{\hr t}(x)$ pour $x\in E$. 

Nous aurons besoin du lemme élémentaire
suivant dont nous donnons la preuve
faute de référence dans la littérature. 

\begin{lemm}\label{lem-boorne-dimension}
Soit $Y\to X$ un morphisme
entre espaces $k$-analytiques, soit
$y$ un point de $Y$ et soit
$x$ son image sur $X$. On a alors
\[\dim_y Y\leq \dim_x X+
\dim_y Y_x.\]
\end{lemm}

\begin{proof}
L'assertion
peut se démontrer après extension
arbitraire du corps de base, ce 
qui permet de supposer
$x$ et $y$ rigides. 
Mais les dimensions en jeu sont 
alors des dimensions de Krull d'anneaux
locaux
noethériens, et le lemme est ramené
à un énoncé classique d'algèbre 
commutative,
\textit{cf.}
\cite[
\href{https://
stacks.math.columbia.edu/tag/00OM}
{Tag 00OM}]{stacks-project}.
\end{proof}

\paragraph{Platitude analytique}
La notion de platitude en géométrie analytique
joue un rôle majeur dans le présent travail. Elle est introduite au chapitre 4 de 
\cite{ducros2018} et y est ensuite étudiée systématiquement. 
On y démontre par exemple le théorème 9.2.1 qui assure que si $\phi \colon Y\to X$ est un morphisme
plat entre
espaces $k$-analytiques avec $Y$ compact et $\Gamma$-strict et $X$ séparé alors $\phi(Y)$ est un domaine analytique compact et $\Gamma$-strict
de $X$ (dans le cas où $\Gamma=\{1\}$, c'est dû à Raynaud, \textit{cf.}
\cite[cor. 5.11]{frg2}),
ou le théorème 10.3.2 qui assure que si $Y\to X$ est un morphisme d'espaces
$k$-analytiques et $\mathscr F$ un faisceau cohérent sur $Y$, le lieu de platitude de $\mathscr F$ sur $X$ est un ouvert de Zariski de $X$. 

Un exemple
particulièrement important de morphismes plats est fourni par les morphismes
\textit{quasi-lisses}, qui sont l'objet chapitre 5 de \cite{ducros2018} -- et parmi ceux-ci on appelle \textit{quasi-étales}
ceux qui sont de dimension relative nulle. Nos définitions reposent sur le critère jacobien et 
sont inspirées par le traitement de la lissité
dans \cite{bosch-l-r1990} (et nous vérifions que notre définition de morphisme quasi-étale
est bien équivalente à celle que donne Berkovich dans \cite{berkovich1993}). Le préfixe «quasi» fait référence à la présence
possible de bord ; un morphisme entre espaces $k$-analytiques (resp. entre bons espaces $k$-analytiques) est étale (resp. lisse)
au sens de \cite{berkovich1993} si et seulement s'il est quasi-étale (resp. quasi-lisse) et sans bord, \textit{cf.} \cite
[cor. 5.4.8, 
rem. 5.4.9 et lemme 5.4.11]{ducros2018}.

\section{Enjoliveurs}\label{s-enjoliveurs}

\textit{Nous fixons pour toute la suite du texte un corps
ultramétrique complet $k$ et un sous-groupe $\Gamma$ de
$\R_+^\times$ tel que $\Gamma\cdot \abs{k^\times}\neq \{1\}$}.

Nous renvoyons le lecteur à la section 5 de \cite{ducros2021a} pour la théorie des éclatements
en géométrie analytique, qui ne pose aucune difficulté -- c'est le décalque 
\textit{mutatis mutandis} de la théorie algébrique.

\begin{defi}
Soit $X$ un espace $k$-analytique
$\Gamma$-strict
compact,
et soit $(Z,S)$ un couple
formé d'un espace $X$-analytique compact
et $\Gamma$-strict
$Z$ et d'un diviseur
de Cartier (effectif) $S$
de $Z$. Soit $n$ un entier. 

Une \textit{présentation $\Gamma$-admissible}
de longueur $n$ de $(Z,S)$ 
est la donnée 
d'une factorisation 
\[Z_{2n}=Z\to Z_{2n-1}\to\cdots \to Z_0=X\]
et, pour tout $i$, d'un
sous-espace analytique fermé
$S_i$ de $Z_i$, ces objets
étant assujettis aux conditions suivantes : 
\begin{enumerate}[a]

\item $S_{2n}=S$ et $S_0=\emptyset$ ; 
\item chacun des $Z_i$ est compact et $\Gamma$-strict ; 
\item si $i$ est impair, $Z_{i+1}\to Z_i$
est un éclatement de centre
$S_i$ et $S_{i+1}$
est son diviseur exceptionnel $X_{i+1}\times_X S_i$ ; 
\item si $i$ est pair, $Z_{i+1}\to Z_i$ est quasi-étale
et
$S_{i+1}$ majore $X_{i+1}\times_X S_i$. 
\end{enumerate}
\end{defi}

Si le groupe $\Gamma$ n'a pas besoin d'être spécifié ou, ce qui revient au même,
s'il est égal à $\R_+^\times$, nous parlerons simplement 
de présentation admissible. 

\begin{rema}\label{rem-prolongement-idiot}
On peut toujours 
rajouter artificiellement un cran à une présentation
$\Gamma$-admissible 
$(Z_{2n}\to Z_{2n-1}\to\cdots \to Z_0,(S_i))$
de $(Z,S)$ 
en posant
$Z_{2n+2}=Z_{2n+1}=Z_{2n}$, en prenant les flèches
$Z_{2n+2}\to Z_{2n+1}$ et $Z_{2n+1}\to Z_{2n}$ 
toutes deux égales à l'identité de $Z_{2n}$, et 
en posant $S_{2n+2}=S_{2n+1}=S_{2n}$. 
\end{rema}

\begin{defi}\label{defi-enjoliveur}
Soit $X$ un espace $k$-analytique compact. 
Un \textit{$\Gamma$-enjoliveur} 
de $X$ 
est un couple $(Z,S)$
où $Z$
est un espace $X$-analytique compact
et $\Gamma$-strict, et où $S$
est un diviseur de Cartier effectif de $Z$, 
qui possède une présentation
$\Gamma$-admissible. 

La \textit{classe}
d'un $\Gamma$-enjoliveur $(Z,S)$ est le plus petit
entier $n$ tel que $(Z,S)$ possède
une présentation admissible de longueur $n$. 

Si le groupe $\Gamma$ n'a pas besoin d'être spécifié ou, ce qui revient au même
s'il est égal à $\R_+^\times$, nous parlerons simplement 
d'enjoliveur. 

\end{defi}

Indiquons quelques propriétés des
enjoliveurs qui découlent de
leur définition et des propriétés connues des éclatements et morphismes 
quasi-étales. Soit donc $X$ un espace $k$-analytique et
soit $(Z,S)$ un enjoliveur de $X$.

Le morphisme $Z\to X$ est quasi-étale en dehors de $S$. 
Si $T$ est un diviseur de Cartier de $X$ alors $Z\times_X T$ est un diviseur de Cartier de $Z$. Si $X$ est réduit 
alors $Z$ est réduit.

Si $X$ est purement de dimension $m$ pour un certain entier $m$ il en va de même de $Z$. En tant que diviseur de Cartier
de $Z$ l'espace $S$ est alors vide si $m=0$ et purement de dimension $m-1$ si $m\geq 1$. 

\begin{lemm}\label{lem-coprod-enjoliveurs}
Soit $X$ un espace
$k$-analytique $\Gamma$-strict
et compact, et soient
$(Z,S)$ et $(Z',S')$ deux
$\Gamma$-enjoliveurs sur $X$, 
de classes respectives $n$ et $n'$. 
La somme disjointe $(Z\coprod Z',
S\coprod S')$ est alors un
$\Gamma$-enjoliveur sur $X$. 
Sa classe est majorée par $\max(n,n')$, 
et l'on a même égalité sauf dans 
le cas où $X$
est non vide, où $Z$ et $Z'$ sont
deux ouverts fermés complémentaires
de $X$, et où $S$ et $S'$ sont vides. 
Dans ce cas 
$(Z\coprod Z',S\coprod S')=(X,\emptyset)$
est de classe nulle, mais $\max(n,n')=1$. 
\end{lemm}

\begin{proof}
Par symétrie
des arguments on peut supposer
$n\geq n'$. 
Choisissons une présentation 
$\Gamma$-admissible $(Z_{2n}\to \cdots \to Z_0, (S_i))$
de $(Z,S)$ et une présentation
$\Gamma$-admissible
$(Z'_{2n'}\to \cdots \to Z'_0, (S'_i))$ de $(Z',S')$. 
Par la remarque \ref{rem-prolongement-idiot}, on peut
prolonger de façon triviale
la présentation 
$\Gamma$-admissible 
$(Z'_{2n'}\to \cdots \to Z'_0, (S'_i))$ en une
présentation
$\Gamma$-admissible 
$(Z'_{2n}\to \cdots \to Z'_0, (S'_i))$ de $(Z',S')$ de longueur 
$n$. Il est alors immédiat que 
\[\left(Z_{2n}\coprod Z'_{2n}\to \cdots \to Z_1
\coprod Z'_1\to X, ((S'_i\coprod S_i)_{1\leq i
\leq 2n}, \emptyset)\right)\] est une présentation
$\Gamma$-admissible
de $(Z\coprod Z',S\coprod S')$. Il s'ensuit que 
le couple
$(Z\coprod Z',S\coprod S')$ est un $\Gamma$-enjoliveur de
$X$ de classe
$\leq n$. 
Nous allons montrer que sa classe est exactement $n$ sauf
dans le cas exceptionnel mentionné dans l'énoncé. 
Soit $(T_{2p}\to \cdots \to T_0, (\Sigma_i))$
une présentation
$\Gamma$-admissible de $(Z\coprod Z',S\coprod S')$.
On distingue maintenant deux cas. 

Supposons tout d'abord que $p=0$, c'est-à-dire que 
$(Z\coprod Z',S\coprod S')$ est de classe nulle. Dans ce
cas $Z$ et $Z'$ sont des ouverts fermés complémentaires
de $X$, et $S$ et $S'$ sont vides. 
Réciproquement, si  $Z$ et $Z'$ sont des
ouverts fermés complémentaires de $X$
et si $S$ et $S'$ sont vides, alors $(Z\coprod
Z',S\coprod S')$ est de classe nulle, et $(Z,S)$
(resp. $(Z'
,S')$
est de classe nulle si $Z=X$
(resp. $Z'=X$) et de classe $1$ sinon.

Supposons maintenant que $p>0$. 
On dispose alors d'un morphisme 
de $Z\coprod Z'$ vers $T_{2p-1}$ qui est un
éclatement de centre $S_{2p-1}$ et de 
diviseur exceptionnel $S\coprod S'$.  
La décomposition de l'espace
éclaté comme somme disjointe
$Z\coprod Z'$ correspond à une
décomposition de la réunion (ensembliste)
des composantes irréductibles de $T_{2p-1}$
non contenues dans $\Sigma_{2p-1}$ comme une union
disjointe $F\coprod F'$ d'ouverts fermés (de sorte
que $Z$ soit l'image réciproque de $F$, et $Z'$ celle 
de $F'$). 
Soit $\mathscr I$ le faisceau cohérent
d'idéaux sur $T_{2p-1}$ correspondant à $\Sigma_{2p-1}$, 
et soit $\mathscr J$ celui correspondant à $F'$ 
(muni disons de sa structure réduite). L'intersection
$\mathscr I\cap \mathscr J$ définit alors un sous-espace
analytique fermé $G$ de $T_{2p-1}$ ; par construction,
l'éclaté de $T_{2p-1}$ le long de $G$ s'identifie à $Z$,
et son diviseur exceptionnel à $S$. 
En conséquence, 
\[(Z\to T_{2p-1}\to \cdots\to T_0, (S, G,
\Sigma_{2p-2},\ldots,\Sigma_0))\]
est une présentation
$\Gamma$-admissible de $(Z,S)$ de
longueur $p$. Il vient $p\geq n$, et la classe
de $(Z\coprod Z',S\coprod S')$ est
dès lors égale à $n$.
\end{proof}

Nous allons maintenant expliquer 
comment un enjoliveur d'un sous-espace
analytique fermé peut se relever
en un enjoliveur de l'espace ambiant. 
Nous aurons pour ce faire besoin de 
deux lemmes.

\begin{lemm}\label{separe-sansbord}
Soit $Y\to X$ un morphisme topologiquement séparé entre espaces $k$-analytiques. 
Ce morphisme
est séparé si et seulement si sa diagonale $\delta\colon Y\to Y\times_X Y$ est sans bord.
\end{lemm}

\begin{proof}
La condition est clairement nécessaire puisqu'une immersion fermée est sans bord. 
Réciproquement, supposons que $\delta$ est sans bord. Pour montrer que 
$Y\to X$ est séparé, on peut raisonner G-localement sur $X$ et donc le supposer affinoïde. 
Dans ce cas $Y$ est topologiquement séparé, et il suffit de démontrer que $Y\to X$
est localement séparé \cite[prop. 1.4.2]{berkovich1993}. Si $U$ est un ouvert de $Y$, la restriction
$\delta|_U$
est composée de la diagonale $U\to U\times_X U$
et de l'immersion ouverte $U\times_X U\hookrightarrow Y\times_X Y$ ; par conséquent, 
la diagonale  $U\to U\times_X U$ est sans bord. On peut donc raisonner sur un recouvrement ouvert de $Y$, 
ce qui permet de supposer ce dernier paracompact. Il possède alors un recouvrement localement fini $(V_i)$
par des domaines affinoïdes, la réunion $W$ des $V_i\times_XV_i$ est un domaine analytique fermé
de $Y\times_X Y$ et $\delta$
se factorise par une immersion fermée $Y\hookrightarrow W$. Comme $W\hookrightarrow
Y\times_X Y$ est localement séparé et comme $\delta$ est sans bord, il résulte de \cite[cor. 5.7]{temkin2004}
que $\delta(Y)\subset \mathrm{Int}(W/Y\times_X Y)$, qui n'est autre
que l'intérieur \textit{topologique}
$W'$ de $W$ dans $Y\times_X Y$. Puisque $\delta(Y)$ est fermé dans $W$, il est fermé dans $Y\times_X Y$, 
et si l'on note $Y'$ son ouvert complémentaire alors $(W',Y')$ est un recouvrement ouvert de $Y\times_X Y$ ; 
par construction, $\delta^{-1}(W')\hookrightarrow W'$ et $\delta^{-1}(Y')\hookrightarrow Y'$ sont des immersions 
fermées, si bien que $\delta$ est une immersion fermée. 
\end{proof}

Nous nous proposons maintenant de démontrer un lemme de prolongement
des morphismes quasi-étales au voisinage d'un sous-espace analytique fermé ; 
dans le cas sans bord, c'est-à-dire dans le cas
des morphismes étales, il est essentiellement dû à Berkovich, 
\emph{cf.} \cite[Prop. 4.3.4]{berkovich1993}, et nous nous ramènerons
à ce dernier résultat. 

\begin{enonce}[remark]{Conventions et notations}
Nous nous servirons de la  notion 
de \emph{germe} d'un espace analytique le long
d'une de ses parties, pour laquelle
nous renvoyons à 
\cite[3.4]{berkovich1993}.

Nous dirons qu'un morphisme d'espaces
$k$-analytiques $f$ est \emph{compact} si
$f^{-1}(K)$ est compact pour toute partie
compacte $K$ du but ; comme un espace
de Berkovich a une base de voisinages
compacts, 
cela revient à demander que $f$ soit
topologiquement séparé et
topologiquement propre.

Soit $X$ un espace $k$-analytique $\Gamma$-strict paracompact.
Nous définissons les catégories suivantes :

\begin{itemize}[label=$\diamond$]

\item On note $\mathsf C(X)$ (resp. $\mathsf C^{\mathrm{comp}}(X)$) 
la sous-catégorie pleine de la catégorie des espaces $X$-analytiques
$\Gamma$-stricts et paracompacts
dont le morphisme structural vers $X$
est séparé et quasi-étale (resp. séparé, quasi-étale et compact).

\item Pour tout sous-espace analytique fermé $Z$ de $X$ nous noterons
$\mathsf D(X,Z)$ 
(resp. $\mathsf D^{\mathrm{comp}}(X,Z)$)
la sous-catégorie pleine de la catégorie des
$(X,Z)$-germes d'espaces $k$-analytiques constituée
des $(X,Z)$-germes isomorphes à un $(X,Z)$-germe de la forme $((Y,f^{-1}(Z)),f)$ où $Y$
est $\Gamma$-strict et paracompact et où 
$f\colon Y\to X$ est un morphisme séparé et quasi-étale (resp. séparé, quasi-étale
et compact). Le plus souvent un objet de $\mathsf D(X,Z)$ ou de $\mathsf D^{\mathrm{comp}}(X,Z)$
sera noté $(Y,T)$ sans référence explicite à son morphisme structural vers $(X,Z)$. 
\end{itemize}

Si $(Y,T)$ est un objet de $\mathsf D(X,Z)$ (resp. $\mathsf D^{\mathrm{comp}}(X,Z)$), le morphisme
structural de $(Y,T)$ vers $(X,Z)$ induit un morphisme de $T$ vers $Z$ qui par définition 
des catégories en jeu fait de $T$
un objet de $\mathsf C(Z)$ (resp. $\mathsf C^{\mathrm{comp}}(Z)$). Le diagramme
commutatif 
\[\begin{tikzcd}
\mathsf D^{\mathrm{comp}}(X,Z)\ar[d,"{(Y,T)\mapsto T}"']\ar[r,hook]&\mathsf D(X,Z)\ar[d,"{(Y,T)\mapsto T}"]\\
\mathsf C^{\mathrm{comp}}(Z)\ar[r,hook]&\mathsf C(Z)
\end{tikzcd}
\]
résume la situation. 
\end{enonce}

\begin{lemm}\label{lemm-quasi-etale-ferme}
Soit $X$ un espace $k$-analytique
$\Gamma$-strict
et
paracompact, et soit $Z$ un sous-espace $k$-analytique fermé de $X$.

\begin{enumerate}[1]
\item Le foncteur $(Y,T)\mapsto T$
de $\mathsf D(X,Z)$ vers $\mathsf C(Z)$ est pleinement fidèle.
\item Le foncteur $(Y,T)\mapsto T$
de $\mathsf D^{\mathrm{comp}}(X,Z)$ vers $\mathsf C^{\mathrm{comp}}(Z)$ est une équivalence
de catégories. 
\end{enumerate}
\end{lemm}

\begin{proof}
Montrons tout d'abord (1). 
Soient $(Y,T)$ et $(Y',T')$ deux objets de $\mathsf D(X,Z)$.
Soit $\tau$
un $Z$-morphisme $T\to T'$. 
Le morphisme
$\tau$
induit une section
$\sigma$ de l'espace $T$-analytique $T'\times_Z T$.
Par séparation, le sous-ensemble
$\sigma(T)$ de $T'\times_Z T$ est contenu dans
l'intérieur relatif de $
T'\times_Z T$ sur $T$. Puisqu'une immersion fermée est sans bord, 
cette image est même contenue dans $\Omega:=\mathrm{Int}
(Y'\times_X Y/ Y)$. Remarquons
que comme le morphisme $
Y'\times_X Y\to Y$ est quasi-étale,
$\Omega$ est exactement l'ensemble des points
en lesquels il est étale.
En vertu de \cite[prop.
4.1.2, cor. 4.1.3 (ii) et prop. 4.3.4]{berkovich1993},
la section $\sigma$ du faisceau
sur $T_{\mathrm{\acute et}}$ représenté par 
$\Omega\times_Y T\to T$ s'étend en une unique
section du faisceau sur $(Y,T)_{\mathrm{\acute et}}$
représenté par $\Omega$; mais comme $Y$ est paracompact, 
une telle section n'est autre qu'un germe de section du
morphisme $\Omega\to Y$ au voisinage de $T$ \cite[prop. 4.3.5]{berkovich1993}.
Il s'ensuit que le morphisme $\tau$ initialement
donné se prolonge en un unique germe de $X$-morphisme
de $Y$ vers $Y'$ au voisinage de $T$, ce qui achève la preuve de
(1). 

Montrons maintenant (2). 
Il résulte de (1) que le foncteur $(Y,T)\mapsto T$ de $\mathsf D^{\mathrm{comp}}(X,Z)$ vers 
$\mathsf C^{\mathrm{comp}}(Z)$
est pleinement fidèle ; il reste à montrer
son essentielle surjectivité. On peut pour cela, en raisonnant composante
par composante, supposer $X$ connexe. Il est alors dénombrable
à l'infini et en particulier réunion d'une famille dénombrable $(X_i)$
de domaines affinoïdes
$\Gamma$-stricts. 
Soit $T$ un espace $k$-analytique 
paracompact et $\Gamma$-strict et soit $f\colon T\to Z$ un morphisme séparé, compact et quasi-étale. 
Pour tout $i$, posons $Z_i=Z
\times_X X_i$ et $T_i=T\times_Z Z_i$. Chacun des $T_i$ est 
$\Gamma$-strict et compact 
car $f$ est compact, et il possède donc lui-même un recouvrement
affinoïde $\Gamma$-strict fini $(T_{ij})_j$.
Fixons $(i,j)$, et soit $t\in T_{ij}$. Comme $T_{ij}$ et $Z_i$ sont bons, il existe un
morphisme étale $U\to Z_i$ et un $Z$-isomorphisme entre un voisinage affinoïde
$\Gamma$-strict
$S$ de $t$ dans $T_{ij}$
et un domaine affinoïde de $U$. Quitte à restreindre $U$ (et $S$) on peut en vertu
de la prop. 4.2.1
de \cite{berkovich1993}
supposer que 
$U$ est de la forme $V\times_{X_i} Z_i$ où $V\to X_i$ est étale et séparé. En tant que domaine analytique compact
et $\Gamma$-strict 
de $U$, on peut écrire $S$ sous la forme $\Sigma\cap U=\Sigma\times_{X_i}Z_i$ pour un certain domaine
analytique compact et $\Gamma$-strict $\Sigma$ de $V$. 
Il découle de ce qui précède qu'il existe un recouvrement dénombrable 
$(S_\ell)_{\ell \in \Z_{\geq 0}}$ de $T$ par des domaines
analytiques compacts
et $\Gamma$-stricts tel que chaque $S_\ell$ s'écrive $\Sigma_\ell \times_X Z$
pour un certain espace $X$-analytique $\Sigma_\ell$ qui est compact et $\Gamma$-strict
et dont le morphisme structural
vers $X$ est quasi-étale et séparé. Nous allons construire récursivement une famille $(\Sigma'_\ell)$
d'espaces $k$-analytiques compacts
et $\Gamma$-stricts munis de morphismes quasi-étales vers $X$ possédant les propriétés suivantes : 
\begin{itemize}[label=$\diamond$]
\item Pour tout $\ell$ chacun des $\Sigma_m$ avec $m\leq \ell$ s'identifie (au-dessus de $X$)
à un domaine analytique de $\Sigma'_\ell$, l'espace $\Sigma'_\ell$ est la réunion des
$\Sigma_m$ pour $m\leq \ell$, et $\Sigma'_\ell\times_X Z$ s'identifie à la réunion
$S'_\ell$ des $S_m$ pour $m\leq \ell$, d'une manière compatible aux plongements
$\Sigma_m\hookrightarrow \Sigma'_\ell$. 

\item Pour tout $\ell$ l'espace $\Sigma'_\ell$ s'identifie (au-dessus de $X$)
à un domaine analytique de $\Sigma'_{\ell+1}$, 
et ce de manière compatible avec
les flèches $\Sigma_m\hookrightarrow \Sigma'_\ell$ et 
$\Sigma_m\hookrightarrow \Sigma'_{\ell+1}$ pour tout $m\leq \ell$. 
\end{itemize}

On pose $\Sigma'_0=\Sigma_0$. Supposons $\ell>0$ et $\Sigma'_\ell$ construit. 
Il existe un domaine analytique
compact et $\Gamma$-strict $F$ de $\Sigma_{\ell+1}$ et un domaine
analytique compact et $\Gamma$-strict $G$ de $\Sigma'_\ell$ tels que
$S_{\ell+1}\cap S'_\ell$
s'identifie à $F\times_X Z$ en tant que domaine analytique de $S_{\ell+1}$, et
à $G\times_X Z$ en tant que domaine analytique de $S'_\ell$. 
Il résulte alors de l'assertion (1) déjà démontrée que quitte à restreindre $F$
et $G$ on peut supposer qu'il existe un $X$-isomorphisme $\iota\colon F\simeq G$ compatible
aux identifications de  $F\times_X Z$ et  $G\times_X Z$ avec $S_{\ell+1}\cap S'_\ell$, et 
l'on définit $S'_{\ell+1}$ comme le recollement de $S_{\ell+1}$ et $S'_\ell$ le long
de $\iota$. 
La famille $(\Sigma'_\ell)$ étant construite, on note $\Sigma$ le $X$-espace obtenu en recollant
les $\Sigma'_\ell$ le long des plongements $\Sigma'_\ell\hookrightarrow \Sigma'_{\ell+1}$. 
Par construction $\Sigma$ est un espace $k$-analytique $\Gamma$-strict muni d'un morphisme 
compact et quasi-étale $\Sigma\to X$ et d'une identification $\Sigma\times_X Z\simeq T$. 
Toutefois, compte-tenu de la manière un peu brutale avec laquelle nous avons procédé aux recollements, 
il est possible que $\Sigma\to X$ ne soit pas séparé. 
Mais puisque $T\to Z$ est séparé, le morphisme diagonal
$\Sigma\times_X\Sigma \to \Sigma$ est sans bord au-dessus de $Z$ ; par compacité des flèches en jeu, 
il est sans bord au-dessus d'un voisinage de $Z$, et $\Sigma\to X$ est alors séparé au-dessus de ce même
voisinage en vertu du lemme \ref{separe-sansbord}, ce qui termine la démonstration. 
\end{proof}

\begin{enonce}[remark]{Notation}
Soit $V\to Z$ un morphisme entre espaces $k$-analytiques, soit $S$ un sous-espace analytique fermé 
de $Z$ et soit $\mathscr F$ un faisceau cohérent sur $V$. Nous noterons $\mathscr F\oslash_S V$ 
le quotient de $\mathscr F$ par son sous-faisceau constitué des sections à support contenu  
ensemblistement dans $V\times_Z S$. C'est encore
un faisceau cohérent, \textit{cf.} \cite[6.1]{ducros2021a}. 

Le quotient $\mathscr O_V\oslash_Z S$ de $\mathscr O_V$ définit un sous-espace analytique fermé
de $V$ que nous noterons $V\oslash_Z S$. C'est par construction l'adhérence analytique de
$V\setminus (V\times_Z S)$ dans $V$  (l'adhérence analytique est le pendant
de l'adhérence schématique en géométrie analytique, \textit{cf.} \cite[lemme-définition 2.7]{ducros2021a}).
\end{enonce}

\begin{prop}\label{prop-enjoliveur-ferme}
Soit $X$ un espace $k$-analytique et $\Gamma$-strict compact et soit
$Y$ un sous-espace analytique fermé de $X$. 
Soit $(W,T)$ un 
$\Gamma$-enjoliveur de $Y$ et soit $n$ sa classe. Il existe un $\Gamma$-enjoliveur $(Z,S)$ de $X$
de classe $\leq n$ 
tel que les propriétés suivantes soient satisfaites : 

\begin{enumerate}[1]

\item $W=(Y\times_X Z)\oslash_Z S$ ;

\item 
$T=W\times_Z S$ ; 
\item l'image de $S$ sur $X$ est égale à celle de $T$. 
\end{enumerate}
\end{prop}

\begin{proof}
On raisonne
par récurrence sur $n$. 
Si $n=0$
alors 
$(W,T)=(Y,\emptyset)$ et on peut prendre
$(Z,S)=(X,\emptyset)$. Supposons
maintenant $n>0$ et le résultat vrai pour les entiers $<n$. 
Puisque $(W,T)$ est de classe $n$ il existe un
$\Gamma$-enjoliveur
$(U,\Sigma)$ de classe $n-1$ sur $Y$, un espace
$k$-analytique compact et $\Gamma$-strict $V$, un
morphisme
quasi-étale $V\to U$, et un sous-espace
analytique fermé $E$ de $V$ contenant
$V\times_U\Sigma$ et tel que
$W$ s'identifie  l'éclatement de $V$
le long de $E$, et $T$ à son diviseur exceptionnel. 
L'hypothèse de récurrence assure alors
l'existence d'un
$\Gamma$-enjoliveur
$(D,R)$ de $X$ de classe $\leq n-1$
tel que les trois propriétés suivantes soient satisfaites : 
\begin{enumerate}[a]

\item $U=(Y\times_X D)\oslash_D R$ ;

\item 
$\Sigma=U\times_D R$ ; 
\item l'image de $R$ sur $X$ est égale à celle de $\Sigma$. 
\end{enumerate}
Il résulte du lemme
\ref{lemm-quasi-etale-ferme}
qu'il existe un espace $k$-analytique compact et $\Gamma$-strict 
$D'$ et un morphisme quasi-étale et séparé
$D'\to D$ tel que $V$
s'identifie à
$D'\times_D U$. 
Soit $\mathscr I$ le faisceau
cohérent d'idéaux de $\mathscr O_D$ définissant $R$ et soit 
$\mathscr J$ le faisceau cohérent
d'idéaux de $\mathscr O_{D'}$ définissant
$E\hookrightarrow V\hookrightarrow D'$ ; notons $F$
le sous-espace analytique fermé
de $D'$ défini par le faisceau d'idéaux
$(\mathscr I
\cdot\mathscr O_{D'})\cap\mathscr J$ ; notons $Z$ l'éclaté
de $D'$ le long de $F$ et $S$
le diviseur exceptionnel de cet éclatement. 
Par construction, $(Z,S)$ est un
$\Gamma$-enjoliveur de $X$ de classe $\leq
n$. 
L'espace $W$ s'identifie à l'éclatement de $V$
le long de $E$. Puisque
$\Sigma=U\times_D R$ et puisque $E$ contient $V\times_U
\Sigma$, le sous-espace analytique 
fermé $V\times_{D'}F$ de $V$ s'identifie à $E$. Par
conséquent, $W$ est la transformée stricte de $V$
relative à l'éclatement $Z\to D'$ si bien 
que $T=W\times_Z S$ et que
\begin{eqnarray*}
W&=&(Z\times_{D'}V)\oslash_Z S\\
&=&(Z\times_D U)\oslash_Z S\\
&=&(Z\times_D ((Y\times_X D)\oslash_D R))
\oslash_ZS\\
&=&
(Z\times_X Y)\oslash_Z S,
\end{eqnarray*}
où la dernière égalité provient du fait 
que $S$ contient par définition $Z\times_D R$. 
L'égalité $T=W\times_Z S$ assure que l'image de $T$ sur $X$
est contenue dans celle de $S$. Et l'image de $S$ sur $X$
est contenue dans celle de $F$, qui est elle-même contenue
dans la réunion de l'image de $E$ et de l'image de $R$, 
donc dans la réunion de l'image de $E$ et de l'image de $\Sigma$. 
Puisque $E$ contient $V\times_U \Sigma$, il s'ensuit que l'image de
$S$ sur $X$ est contenue dans l'image de $E$, et partant dans celle
de $T$. Par conséquent, l'image de $S$ sur $X$ est finalement
égale à celle 
de $T$, ce qui achève la démonstration. 
\end{proof}

 \section{Aplatissement d'un faisceau cohérent : 
 rappels et compléments}\label{s-aplatissement}

Le but de cette section est 
d'énoncer le
théorème
principal de \cite{ducros2021a}, ou
plus précisément une version
un peu renforcée de ce dernier, mais
qu'on peut en fait essentiellement 
déduire de sa preuve, comme nous
allons l'expliquer ensuite. 

\begin{enonce}[remark]{Notations}
Soit $f\colon Y\to X$ un morphisme d'espaces $k$-analytiques
et soit $\mathscr F$ un faisceau cohérent sur $Y$. 
Nous noterons $\pl {\mathscr F}X$ le lieu de platitude
relative de $\mathscr F$ sur $X$. C'est un ouvert de Zariski
de $Y$ dont le fermé complémentaire sera noté
$\ql {\mathscr F}X$. Pour tout entier $n$, nous
noterons $\qln {\mathscr F}Xn$ l'ensemble
des points de $\ql {\mathscr F}X$ 
en lesquels ce dernier est de dimension relative 
$\geq n$ sur $X$. On dit que $\mathscr F$ est 
\textit{$X$-plat en dimensions supérieures ou égales à $n$}
si $\qln {\mathscr F}Xn$ est vide. 
Enfin, nous noterons
$\qlns {\mathscr F}Xn$ le
saturé de $\qln {\mathscr F}Xn$ pour la relation d'équivalence
définie par $f$, c'est-à-dire $f^{-1}(f(\qln {\mathscr F}Xn))$.
Lorsque $\mathscr F=\mathscr O_Y$ nous écrirons 
$\pl YX$, etc. au lieu de $\pl {\mathscr O_Y}X$, etc. 
\end{enonce}

\begin{theo}[version un
peu renforcée du théorème 6.6
de \cite{ducros2021a}]\label{theo-aplatissement-bis}

Soit $f\colon Y\to X$ un morphisme entre espaces
$k$-analytiques compacts et $\Gamma$-stricts.
Soient $N$ et $n$ deux entiers avec $n\leq N$ et soit
$\mathscr F$ un faisceau cohérent sur $Y$ 
qui est plat en dimensions $\geq N$. 
Il existe alors un 
$\Gamma$-enjoliveur $(Z,S)$ de $X$
de classe $\leq N-n$ et un domaine analytique compact et 
$\Gamma$-strict $V$ de $Y\times_X Z$ tels que les
conditions suivantes soient satisfaites.

\begin{enumerate}[1]
\item $\mathscr F_V\oslash_Z S$ est $Z$-plat
en dimensions $\geq n$ ; 
\item $V\setminus V\times_Z S$ est contenu dans
l'image réciproque de $Y\setminus
\qln {\mathscr F}Xn$ ; 

\item $V\times_Z S$ est contenu dans l'image réciproque
de $\qlns {\mathscr F}Xn$ ; 

\item l'image de $V\to Y$ contient
$Y\setminus \qlns {\mathscr F}Xn$ ; 

\item l'image de $V\times_Z S\to Y$ contient
tout point de $\qln {\mathscr F}Xn$ dont l'image sur $X$
n'est pas adhérente à $f(\qln {\mathscr F}Xn)\cap \mathsf A(X)$ ; 

\item l'image de $S$ sur $X$ est 
contenue dans celle de $\supp
F$, et
si $\mathscr F\neq 0$ cette
image est 
de dimension $\leq \dim \supp
F-n$.

\end{enumerate}

\end{theo}

L'énoncé ci-dessus comprend deux
améliorations mineures par rapport
à celui du théorème 
6.6 de \cite{ducros2021a}, et une
vraie nouveauté, à savoir l'assertion (6).

Nous allons tout d'abord
discuter brièvement
des deux améliorations mineures.
La première est la suivante : nous aplatissons
$\mathscr F$ (en dimensions $\geq n$)
au moyen
\textit{d'un seul}
$\Gamma$-enjoliveur, alors
que le théorème
6.6 
de \cite{ducros2021a} le fait au moyen
d'une \textit{famille
finie} de
$\Gamma$-enjoliveurs ; mais cette
différence entre les deux énoncés
est illusoire, puisqu'il suffit de remplacer
la famille finie d'enjoliveurs du 
 théorème
6.6 de  
de \cite{ducros2021a} par
sa somme disjointe
(qui est bien un enjoliveur
de la classe requise d'après
le lemme \ref{lem-coprod-enjoliveurs}). 
La seconde amélioration concerne la classe
du $\Gamma$-enjoliveur $(X,S)$. Nous la majorons
par $N-n$ où $N$ est 
\textit{n'importe quel} entier tel que 
$\mathscr F$ soit $X$-plat
en dimensions $\geq N$, 
alors que dans le théorème 6.6 
de \cite{ducros2021a} nous établissons
cette majoration pour un entier $N$
spécifique, à savoir $d+1$ où $d$ est un
entier majorant la dimension des fibres de
$\supp F\to X$ (notons que la $X$-platitude
de $\mathscr F$ en
dimensions $\geq d+1$ est tautologique). 
Or la preuve procède par récurrence
descendante sur $n$, en partant de $n=d+1$, 
et la seule propriété de l'entier $d+1$
qui est utilisée est précisément la 
$X$-platitude de $\mathscr F$ en dimensions
$\geq d+1$, qui rend trivial le cas où
$n=d+1$ (l'enjoliveur
$(X,\emptyset)$ et le domaine
analytique $V=Y$ conviennent alors)
et assure ainsi l'initialisation 
de la récurrence. Il est dès lors
licite dans l'énoncé de remplacer
$d+1$ par n'importe quel entier $N$
tel que 
$\mathscr F$ soit $X$-plat
en dimensions $\geq N$, avec exactement
la même démonstration, à ceci
près que l'on commence par traiter
le cas trivial où $n=N$. 

Venons-en maintenant à l'assertion (6). 
Comme nous l'avons dit, celle-ci ne figure
pas dans notre énoncé originel, mais nous
allons voir qu'elle se déduit relativement
aisément de sa preuve. 
Cette dernière construit en fait étape
par étape les termes $(Z_i,S_i)$
d'une présentation
$\Gamma$-admissible
de $(Z,S)=(Z_{2N-2n},S_{2N-2n})$ en 
partant de $(Z_0,S_0)=(X,\emptyset)$.
De plus elle fournit pour tout entier
$i$ entre $0$ et $N-n$ un domaine
analytique compact et $\Gamma$-strict
$V_{2i}$ tel que $(Z_{2i}, S_{2i}, V_{2i})$
satisfasse aux conclusions du théorème 
pour l'entier $N-i$ (au lieu de $n$), 
avec bien entendu $V_0=Y$ et $V_{2N-2n}=V$. 
Fixons $i<N-n$ et supposons
$Z_{2i}, S_{2i}$ et $V_{2i}$
construits. Le cœur de la preuve
consiste à fabriquer 
$Z_{2i+1}$ et $S_{2i+1}$
(d'où $Z_{2i+2}$ et $S_{2i+2}$
se déduisent par éclatement), ainsi que 
le domaine $V_{2i+2}$. 
Nous n'allons pas redécrire en détail 
l'ensemble du procédé, mais simplement 
évoquer les faits qui vont être utiles
pour établir (6). 
Le sous-espace
analytique fermé
$S_{2i+1}$ de
$Z_{2i+1}$ est défini par
le produit
du faisceau
d'idéaux inversible
définissant $Z_{2i+1}\times_{Z_{2i}}
S_{2i}$ et d'un faisceau d'idéaux
$\mathscr J_i$ dont on note $\Sigma_i$
le sous-espace analytique fermé
associé ; ensemblistement, on a donc
l'égalité $S_{2i+1}= Z_{2i+1}\times_{Z_{2i}}
S_{2i}
\cup \Sigma_i$. De plus, il existe :
\begin{itemize}[label=$\diamond$]
\item une
décomposition de $S_{2i+1}$ comme somme disjointe $\Omega_i\coprod
\Omega'_i$ de deux ouverts fermés ; 
\item un espace $k$-analytique
$\Gamma$-strict compact $\Theta_i$ 
muni d'un morphisme quasi-lisse 
$\Theta_i\to \Omega_i$ dont les fibres
sont géométriquement intègres de
dimension $N-i-1$ ; 
\item un sous-espace analytique fermé
réduit $F_i$ de $\Theta_i$, qui est contenu
dans l'image de la composée d'une famille
finie de $Z_{2i+1}$-morphismes finis, 
composée
dont le but est $\Theta_i$
et dont la source s'immerge dans 
$\mathrm{Supp}(\mathscr F_{V_{2i}
\times_{Z_{2i}}Z_{2i+1}}),$
\end{itemize}
tels que $\mathscr J_i
\mathscr O_{\Omega_i'}=
\mathscr O_{\Omega_i'}$
et que $\mathscr J_i
\mathscr O_{\Omega_i}$ soit l'idéal
des coefficients associé au sous-espace
$F_i$ du $\Omega_i$-espace $\Theta_i$ 
(pour la définition 
et l'existence de l'idéal des coefficients
dans la situation considérée ici, 
voir \cite[th. 3.9]{ducros2021a}). 
Soit $x$ un point de $\Sigma_i$
non situé sur  $Z_{2i+1}\times_{Z_{2i}}
S_{2i}$ ; notons qu'il appartient
en particulier au lieu quasi-étale de
$Z_{2i+1}$ sur $X$. 
Par ce qui précède
$x$ est situé sur $\Omega_i$
et il appartient au lieu des zéros
de l'idéal des coefficients
de $F_i$, ce qui signifie que
$F_{i,x}=\Theta_{i,x}$. 
Puisque $\Theta_{i,x}$ est de
dimension $N-i-1$, 
le point $x$ possède un antécédent $y$
sur $\Theta_i$ tel que $d_{\hr x}(y)=N-i-1$. 
Étant situé sur $F_i$, le point $y$ est 
l'image d'un point $z$ du support
de $\mathscr F_{Z_{2i+1}\times_
{Z_{2i}}V_{2i}}$
par une composée de morphismes
finis, si bien que 
$d_k(y)=d_k(z)$. 
Comme $x$ est situé
sur le lieu quasi-étale
de $Z_{2i+1}\to X$, 
le point $z$ est situé sur le lieu 
quasi-étale de $Z_{2i+1}\times_
{Z_{2i}}V_{2i}\to Y$ ; si $t$
désigne son image sur $Y$ on a donc
$d_k(z)=d_k(t)\leq \dim \supp
F$. Il s'ensuit que l'image 
$\xi$ de $x$ sur $X$ appartient à l'image
de $\supp F$ et que
\[d_k(\xi)=d_k(x)\leq \dim
\supp F-N+i+1.\] 
Soit $s$ un point de $S$. Puisque $S_0$
est vide et puisque $S_{2i}$ est
pour tout $i$ entre $1$ et $N-n$ l'image 
réciproque de $S_{2i-1}$ sur $Z_{2i}$, 
il existe $i$ entre $0$ et $N-n-1$
tel que l'image de $s$ sur $Z_{2i+1}$
appartienne à $S_{2i+1}$ mais pas
à l'image réciproque de $S_{2i}$ ; cette
image est donc située sur
\[\Sigma_i\setminus (Z_{2i+1}\times_{Z_{2i}}
S_{2i}),\] et il résulte alors
de ce qui précède que l'image $\xi$ de $s$
sur $X$ appartient à l'image de $\supp 
F $ (qui est dès lors non vide)
et qu'on a
\[d_k(\xi)\leq \dim \supp F
-N+i+1 \leq \dim \supp F -n.\]

\section{Comment enjoliver un morphisme 
arbitraire}\label{s-comment}

Dans \cite{ducros2021a}, nous utilisons notre théorème
d'aplatissement pour enjoliver de différentes façons un morphisme
entre espaces analytiques, par changement
de base le long d'un enjoliveur et passage à la transformée
stricte (\textit{op. cit.}, théorèmes 7.3 et 7.5). Nous
nous proposons dans cette section 
de reprendre les méthodes de démonstration
de ces énoncés et de les améliorer 
pour obtenir un résultat 
qui les coiffe et les étend (théorème \ref{theo-embellissement}
ci-dessous).

\begin{lemm}\label{proprietes-enjoliveurs}
Soit $Y\to X$ un morphisme entre espaces $k$-analytiques, 
soit $(Z,S)$ est un  enjoliveur de $X$, 
et soit 
$V$ un domaine analytique de $Y\times_X Z$. 
\begin{enumerate}[1]
\item La flèche $V\oslash_Z S\to Y$ est génériquement
quasi-étale, et $V\oslash_Z S$ est réduit dès que $Y$ est réduit.
\item Si $Y$ est purement de dimension $n$ pour un certain entier $n$,
il en va de même  de  $V\oslash_Z S$.
\item Supposons qu'il existe un entier $d$ tel que 
$Y\to X$ soit génériquement de dimension
$d$. Alors $V\oslash_Z S
\to Z$ est génériquement de dimension $d$. 
\end{enumerate}
\end{lemm}

\begin{proof}
L'ouvert $(Y\times_X Z)\setminus (Y\times_X S)$ est quasi-étale sur $Y$, 
et $V\setminus (V\times_Z S)$ est en particulier réduit dès que $Y$ est réduit ; par
conséquent, $V\oslash_Z S\to Y$ est génériquement
quasi-étale, et $V\oslash_Z S$ est réduit dès que $Y$ est réduit, d'où (1). 

Si $v$ est un point de $\mathsf A(V\oslash_Z S)$ alors
$V\oslash_Z S\to Y$ est quasi-étale en $v$ par ce qui précède, si
bien que l'image
$y$
de
$v$  sur $Y$ appartient à $\mathsf A(Y)$ et vérifie l'égalité $d_k(v)=d_k(y)$. 
Il en résulte que si $Y$ est purement de dimension $n$ pour un certain entier $n$,
il en va de même  de  $V\oslash_Z S$, d'où (2). 

Montrons enfin (3). Soit $v$ un point
de $\mathsf A(V\oslash_Z S)$. Comme $v$
est situé par ce qui précède u-dessus d'un point de $\mathsf A(Y)$, 
la dimension de $Y\times_X Z \to Z$ en $v$ est égale à $d$, et il en va de même
de la dimension de $V\to Z$ en $v$ ; et puisque 
$v$ appartient à $\mathsf A(V\oslash_Z S)$, il possède un voisinage ouvert dans $V$
qui est aussi un ouvert de $V\oslash_Z S$, si bien que la dimension de $V\oslash_Z S
\to Z$ en $v$ est encore égale à $d$, ce qui conclut. 
\end{proof}

Le lemme suivant récapitule un certain nombre de faits qui nous seront utiles 
à propos de la dimension relative, et qui se trouvent déjà essentiellement
dans \cite[\S 7]{ducros2021a}.

\begin{lemm}
\label{complements-dim}
Soit $f\colon Y\to X$ un morphisme entre espaces $k$-analytiques. 

\begin{enumerate}[1]
\item 
Pour tout $y\in \mathsf A(Y)$
on a $d_k(f(y))=d_k(y)-\dim_y f.$

\item Soit $y$ un point de $\mathsf A(Y)$ en lequel $f$ est plat. 
Le point $f(y)$ appartient alors à $\mathsf A(X)$. 

\item Si $X$ est réduit et si $x\in \mathsf A(X)$ alors $f$ est plat en tout point 
de $f^{-1}(x)$. 

\item Supposons que $Y$ est non vide soit $d$ la valeur minimale de $\dim f$. 

\begin{enumerate}[b]
\item 
Si $\{y\in Y,\dim_y f\geq d+1\}$ est d'intérieur vide dans $Y$ il coïncide avec
 $\qln YX{d+1}$. 

\item 
Supposons que $Y$ est de dimension finie $n$ et que 
la dimension de $\{y\in Y,\dim_y f\geq d+1\}$ est $<n$. 
Alors $\dim f(Y)=n-d$ et $f^{-1}(f(y))$ est purement de dimension $d$
pour tout $y\in Y$ tel que $d_k(y)=n$. 
\end{enumerate}
\end{enumerate}
\end{lemm}

\begin{proof}
L'assertion (1) est simplement \cite[Lemma
7.1 (1))]{ducros2021a}. 

Soit $y$ un point de $\mathsf A(Y)$ en lequel $f$ est plat et soit 
$d$ la dimension relative de $f$ en $y$. Par (1) on a 
$d_k(f(y))=d_k(y)-d$, et par platitude on a 
$\dim_{f(y)} X=\dim_y Y-d=d_k(y)-d$, si bien que $\dim_{f(y)} X=d_k(f(y))$ ; autrement dit, 
$f(y)\in \mathsf A(X)$, d'où (2). 

L'assertion (3) est \cite[Th. 10.3.7]{ducros2018} ; l'assertion (4a) 
est \cite[Lemme 7.2]{ducros2021a}. 

Montrons enfin (4b). 
Soit $x$ est un point 
appartenant à $f(Y)$. La fibre $f^{-1}(x)$ est non vide
et partout de dimension
supérieure ou égale à $d$ si bien qu'il existe
$z\in f^{-1}(x)$ tel que $d_{\hr x}(z)\geq d$, et
il vient 
\[n\geq d_k(z)=d_{\hr x}(z)+d_k(x)\geq d_k(x)+d.\] En conséquence $d_k(x)\leq n-d$, 
d'où la majoration $\dim f(Y)\leq n-d$. Et observons que si $d_k(x)=n-d$ 
on a nécessairement  $d_{\hr x}(z)=d$ pour tout $z$ comme 
ci-dessus, si bien que $f^{-1}(x)$ est purement de dimension $d$. 
Par ailleurs, comme $\dim Y=n$ l'ensemble
des points $y$ de $Y$ tels que $d_k(y)=n$ est non vide. 
Soit $y$ un tel point. 
Puisque le lieu des points en lesquels $\dim f\geq d+1$ est de dimension 
strictement inférieure à $n$ 
il ne contient pas $y$, ce qui entraîne que 
$\dim_y f=d$, et partant en vertu de (1)
que 
$d_k(f(y))=d_k(y)-\dim_y f=n-d$, si bien que 
$f(Y)$ est finalement exactement de dimension $n-d$ ; et 
par ailleurs
l'égalité $d_k(f(y))=n-d$ implique comme observé plus haut
que la fibre $f^{-1}(f(y))$ est purement de dimension $d$. 
\end{proof}

\begin{theo}\label{theo-embellissement}
Soit $Y\to X$ un morphisme d'espaces
$k$-analytiques.
On suppose que $X$ et $Y$ sont compacts,
$\Gamma$-stricts
et réduits, et 
que $Y$ est non vide et équidimensionnel ;
on note
$n$ sa dimension. On désigne par $d$
et $D$ les valeurs minimale
et maximale de la dimension 
relative de $Y$ sur $X$, et l'on suppose
que le lieu des points en lesquels
cette dimension vaut $d$ est dense.

\begin{enumerate}[A]
\item Il existe un
$\Gamma$-enjoliveur $(Z,S)$ sur $X$ de classe $\leq D+1$, 
et un domaine
analytique compact et $\Gamma$-strict $V$
de $Y\times_X Z$ possédant les propriétés suivantes :

\begin{enumerate}[2]

\item La flèche $V\oslash_Z S\to Y$
est surjective. 

\item La flèche $V\oslash_Z S\to Z$ se factorise par un morphisme
plat et surjectif sur un sous-espace analytique fermé réduit $F$ de $Z$
purement de dimension $n-d$. 

\item L'image de $S$
sur $X$ est contenue dans celle 
de $Y$ et est de dimension 
strictement inférieure à $n-d$.

\end{enumerate}

\item Pour tout $(Z,S, V,F)$ possédant les propriétés
ci-dessus l'espace $V\oslash_Z S$ est réduit et purement de dimension $n$, la flèche  
$V\oslash_Z S\to Y$ est génériquement quasi-étale, la flèche 
$V\oslash_Z S\to F$ est purement
de dimension relative $d$, et $S\cap F$ est d'intérieur vide dans $F$. 
\end{enumerate}

\end{theo}

\begin{proof}
Commençons par montrer l'assertion
(B). 
Il résulte du lemme \ref{proprietes-enjoliveurs}
que $V\oslash_Z S\to Y$ est génériquement quasi-étale et que
$V\oslash_Z S$ est réduit et purement de dimension $n$
(ce dernier point provenant du fait que $Y$ est lui-même purement de dimension $n$), 
et  il découle alors de la platitude de 
$V\oslash_ZS\to F$ que ses fibres sont purement de dimension
$d$. Enfin, soit $s\in S\cap F$.  La fibre de $V\oslash_Z S$ en $s$ est non vide par surjectivité 
de $V\oslash_Z S\to F$, et elle est purement de dimension $d$. Elle possède donc un point
$t$ tel que $d_{\hr s}(t)$ soit égal à $d$. Mais par définition de $V\oslash_Z S$, l'image réciproque de $S$
sur ce dernier est d'intérieur vide ; en conséquence, $n>d_k(t)=d_k(s)+d$, et $d_k(s)<n-d$. 
Il en résulte
que $S\cap F$ est d'intérieur vide dans $F$. 

Il reste donc désormais à démontrer (A).
Nous allons tout d'abord décrire deux
situations
particulières dans lesquelles il suffit, 
pour démontrer l'énoncé (A)
pour le morphisme $Y\to X$, de le démontrer
pour un morphisme auxiliaire. Cela jouera
un rôle crucial dans notre preuve en
permettant de se réduire par étapes
au cas génériquement plat.

\textit{Première situation particulière.}
Supposons qu'il existe un sous-espace analytique fermé 
et réduit $X_0$
de $X$ tel que 
le morphisme 
$Y\to X$ admette une factorisation de
la forme $Y\to X_0\hookrightarrow X$, et
tel que l'assertion (A) vaille
pour le morphisme $Y\to X_0$ ; elle vaut alors pour $Y\to X$.

En effet, il existe par
hypothèse un
$\Gamma$-enjoliveur $(W,T)$ de $X_0$
de classe
inférieure ou égale à $D+1$ et
un domaine analytique compact et $\Gamma$-strict
$U$ de $Y\times_{X_0}W$
satisfaisant aux conditions suivantes.  

\begin{enumerate}[i]

\item La flèche $U\oslash_W T\to Y$ est surjective. 

\item  La flèche $U\oslash_W T\to W$ se factorise par un morphisme
plat et surjectif sur un sous-espace analytique fermé réduit de $W$
purement de dimension $n-d$. 

\item L'image de $T$ sur $X_0$ est contenue dans celle de $Y$ 
et
de dimension $<n-d$. 

\end{enumerate}
En vertu de la
proposition \ref{prop-enjoliveur-ferme},
il existe un
$\Gamma$-enjoliveur $(Z,S)$ de $X$ 
de classe
$\leq D+1$ tel que $W$ s'identifie
à $(X_0\times_X Z)\oslash_ZS$,
tel que $T=W\times_Z S$, et tel que
l'image de $S$ sur $X$ soit égale à celle de $T$.
Choisissons un domaine analytique compact et $\Gamma$-strict
$V$ de $Y\times_X Z$ dont la trace sur
son sous-espace analytique fermé $Y\times_X W
=Y\times_{X_0}W$ est égale à $U$. On a alors
\[V\oslash_Z S=U\oslash_Z S=U\oslash_W T\]
où la première égalité vient du fait que 
$W=(X_0\times_X Z)\oslash_ZS$, et la seconde du fait
que $T=W\times_Z S$. Le triplet $(Z,S,V)$ satisfait alors
l'assertion (A) relativement au morphisme 
$Y\to X$.

\textit{Seconde situation particulière.}
Soient $c$ et $c'$ deux entiers
positifs ou nuls
de somme $D-d$, soit
$(W,T)$ un
$\Gamma$-enjoliveur sur $X$
de classe $c$, 
et soit $U$
un domaine analytique compact et $\Gamma$-strict
de $Y\times_X W$. 
On suppose que $U\oslash_W T \to Y$ est surjective,
que
l'image de $T$ sur $X$
est contenue 
dans celle de $Y$, 
que la dimension relative de
$U\oslash_W T\to W$ est majorée
par $d+c'$ et
que $U\oslash_W T\to W$ 
satisfait l'assertion $(\mathsf A)$
(on 
sait d'après
les remarques préliminaires en début 
de section \ref{s-comment} que 
$U\oslash_W T$
est purement de dimension $n$
et que la dimension 
générique de $U\oslash_W T\to W$
est égale à $d$ ; quant
à $W$, il est
réduit). 
Nous allons montrer que $Y\to X$ satisfait également
$\mathsf A$. 
Dire que $U\oslash_W T\to W$ 
satisfait l'assertion $(\mathsf A)$
signifie qu'il existe un
$\Gamma$-enjoliveur
$(Z,S)$ de $W$
de classe $d+c'+1$
et un domaine analytique compact
et $\Gamma$-strict $V$
de $(U\oslash_W T)\times_W Z$
tels que les conditions suivantes
soient satisfaites 

\begin{enumerate}[i]

\item La flèche $V\oslash_Z S\to U\oslash_W T$ est
surjective. 
\item  La flèche
$V\oslash_Z S\to Z$ se factorise par un morphisme
plat et surjectif sur un sous-espace
analytique fermé réduit de $Z$
purement de dimension $n-d$. 

\item L'image de $S$
sur $W$ est contenue 
dans celle de $U\oslash_W T$
et de dimension $<n-d$.

\end{enumerate}
Soit $\mathscr I$ le faisceau cohérent d'idéaux sur $Z$
correspondant à $S$, et soit $\mathscr J$ celui correspondant
à $Z\times_W T$. Soit $Z'$ l'éclaté
de $Z$ le long de $\mathscr I\cdot \mathscr J$ et
soit $S'$ le diviseur exceptionnel correspondant. 
Comme $Z\to W$ est composé de morphismes
quasi-étales et d'éclatements, le produit fibré
$Z\times_W T$
est un diviseur 
de Cartier de $Z$ ;
par conséquent l'idéal $\mathscr I
\cdot \mathscr J$ est inversible,
si bien que $Z'$ s'identifie 
à $Z$, et $S'$ à la somme des
diviseurs de Cartier effectifs
$S$ et $Z\times_W T$ ;
nous utilisons désormais la notation $Z$
plutôt que $Z'$, mais conservons
la notation $S'$. Par
construction, $(Z,S')$
est un
$\Gamma$-enjoliveur de $X$,
de classe
majorée par $d+c+c'+1=D+1$. 
Soit $s$ un point de $S'$,
soit $w$
son image sur $W$ et soit $x$ l'image de $s$
sur $X$ (qui est aussi l'image de $w$). 
Si $w$ appartient à $T$ alors $x$
appartient à l'image de
$Y$ et $d_k(x)
<n-d$ par hypothèse ; sinon, $s$
appartient à $S$, auquel cas on 
$w$ appartient à l'image 
de $U\oslash_Z T$
et vérifie la
majoration $d_k(w)<n-d$,
si bien que $x$ 
appartient à l'image 
de $Y$ et que $d_k(x)<n-d$.
Le domaine analytique 
$V$ de $(U\oslash_WT)\times_W Z$
est la trace d'un domaine
analytique
compact et $\Gamma$-strict
$\Omega$ de $U\times_W Z$, qu'on peut
également voir comme un
domaine
analytique compact et
$\Gamma$-strict
de $Y\times_X Z$. Une composante
irréductible de $\Omega$ est contenue
dans 
l'image réciproque de $S'$ si et seulement si elle est
contenue dans $Z\times_W T$ ou dans
l'image réciproque de $S$ ; 
par conséquent, $\Omega\oslash_ZS'$
coïncide avec $V\oslash_Z S$. 
Il s'ensuit que $\Omega\oslash_Z S'
\to Y$ est surjective et que 
 $\Omega\oslash_Z S'
\to Z$ se factorise par un morphisme
plat
et surjectif sur un sous-espace
analytique fermé réduit de $Z$
purement de dimension $n-d$. 
Le $\Gamma$-enjoliveur 
$(Z,S')$ de $X$
et le domaine analytique compact
et $\Gamma$-strict $\Omega$ de
$Y\times_X Z$ satisfont les assertions 
(1), (2) et (3) relativement au morphisme $Y\to X$, qui
satisfait dès lors lui-même l'assertion (A). 

Nous pouvons maintenant en venir à la démonstration 
proprement dite.
Nous allons nous réduire
dans un premier temps
au cas relativement équidimensionnel, 
puis dans un second temps
au cas génériquement plat. 

Commençons par la réduction au cas 
où $D=d$.
Appliquons le
théorème \ref{theo-aplatissement-bis}
en prenant le triplet
$(\mathscr F,d,n)$
de son énoncé égal à 
$(\mathscr O_X, D,d+1)$.
Il assure l'existence d'un
$\Gamma$-enjoliveur $(Z,S)$
de classe
$(D-d)$ de $X$ et d'un domaine analytique compact
et $\Gamma$-strict $V$ de $Y\times_X Z$
tel que les propriétés suivantes
soient satisfaites : 

\begin{itemize}[label=$\diamond$]

\item $V\oslash_Z S \to Z$ est plat
en dimensions $\geq d+1$.  
\item L'image de
$V\oslash_Z S\to Y$ contient
$Y\setminus \qlns YX{d+1}$. 

\item 
L'image de $S$ sur $X$ est contenue
dans l'image de $Y$ et de
dimension $\leq 
n-d-1$. 
\end{itemize}
Comme $V\oslash_Z S\to Z$ est génériquement
de dimension $d$, le fait que ce
morphisme soit plat en dimensions 
$\geq d+1$ signifie simplement
qu'il est purement
de dimension relative $d$
(lemme \ref{complements-dim} (4a)).
Et il résulte également
de ce lemme
que 
$\qln YX{d+1}$ est le lieu des points 
de $Y$ en lesquels la dimension relative
de $Y\to X$ vaut $d+1$. L'image 
de $V\oslash_Z S\to Y$ contient donc
toutes les fibres de $Y\to X$ purement
de dimension $d$,
et en
particulier du
lemme au vu  du
lemme \ref{complements-dim} (4b)
tous les points de $\mathsf A(Y)$. 
Cette image est donc dense dans $Y$. 
Comme elle est compacte, c'est $Y$ tout
entier. 
Il résulte dès lors
de la seconde situation 
particulière étudiée plus haut
(en prenant $c=D-d$ et $c'=0$) 
qu'il suffit
de démontrer le théorème pour le morphisme
$V\oslash_Z S\to Z$. On s'est donc bien
ramené au cas où $D=d$, c'est-à-dire
encore à celui où le morphisme
$Y\to X$
est purement de dimension
$d$.

Nous allons maintenant nous ramener au
cas génériquement plat. 
Le morphisme $Y\to X$ est purement de dimension relative $d$. Il s'ensuit
en vertu du corollaire 4.7 de \cite{ducros2007} qu'il existe : 

\begin{itemize}[label=$\diamond$] 
\item un recouvrement fini $(X_i)$ de $X$ par des domaines 
affinoïdes $\Gamma$-stricts ; 
\item pour chaque $i$, un recouvrement fini $(Y_{ij})_j$ de 
$Y\times_X X_i$ par des domaines affinoïdes $\Gamma$-stricts ; 
\item pour chaque $(i,j)$, une factorisation de
la flèche $Y_{ij}\to X_i$
de la forme $Y_{ij}\to T_{ij}\to X_i$ où $T_{ij}$ est un
espace $k$-affinoïde
compact et $\Gamma$-strict, où $T_{ij}\to X_i$ est quasi-lisse
et purement de dimension relative $d$, et où $Y_{ij}\to T_{ij}$
est fini. 
\end{itemize}
Fixons $(i,j)$. Le morphisme $T_{ij}\to X_i$ étant quasi-lisse, il est en particulier
plat et à fibres géométriquement réduites. Le théorème 3.11
de \cite{ducros2021a} assure alors l'existence d'un espace $k$-affinoïde
et $\Gamma$-strict $X_{ij}$ muni d'un morphisme quasi-étale surjectif
$X_{ij}\to X$, et d'un recouvrement $(T_{ij\ell})_\ell$ de
$T_{ij}\times_{X_i}X_{ij}$ par des domaines analytiques compacts et
$\Gamma$-stricts tels que les fibres du morphisme $T_{ij\ell}\to X_{ij}$ soient
géométriquement connexes (et donc ici géométriquement intègres
dès qu'elles sont non vides). Pour tout $\ell$, notons $Y_{ij\ell}$
l'image réciproque de $T_{ij\ell}$ sur $Y_{ij}$, et $X_{ij\ell}$ l'image
de $T_{ij\ell}$ sur $X_{ij}$ (qui en est
par platitude un domaine
analytique compact et $\Gamma$-strict).
Récapitulons : on dispose d'une famille finie $(X_{ij\ell})$ d'espaces $k$-analytiques
compacts et $\Gamma$-stricts et, pour chacun d'eux, d'un domaine analytique compact et $\Gamma$-strict
$Y_{ij\ell}$ de $Y\times_X X_{ij\ell}$ tel que $Y_{ij\ell}\to X_{ij\ell}$
admette une factorisation
\[Y_{ij\ell}\to T_{ij\ell}\to X_{ij\ell}\] où $T_{i\ell}$ est compact et $\Gamma$-strict, 
où $Y_{ij\ell}\to T_{ij\ell}$ est fini, et où $T_{ij\ell}\to X_{ij\ell}$
est quasi-lisse à fibres géométriquement intègres de dimension $d$. Par construction, la réunion des
images des $Y_{ij\ell}\to Y$ est égale à $Y$ tout entier. 
Fixons $(i,j,\ell)$. Le morphisme $Y_{ij\ell}\to T_{ij\ell}$ étant fini, 
son image $F_{ij\ell}$ est un fermé de Zariski de $T_{ij\ell}$. Puisque $Y_{ij\ell}\to X_{ij\ell}$ est 
purement de dimension relative $d$ et puisque les fibres de $T_{ij\ell}\to X_{ij\ell}$
sont irréductibles et de dimension $d$, le fermé $F_{ij\ell}$ est une réunion de fibres de
$T_{ij\ell}\to X_{ij\ell}$. L'existence d'un idéal des coefficients pour $F_{ij\ell}$ (muni par exemple de
sa structure réduite), assurée par le théorème 3.9 de \cite{ducros2021a}, assure alors que l'image
de $F_{ij\ell}$ sur $X_{ij\ell}$ est un fermé de Zariski $\Lambda_{ij\ell}$
de ce dernier. Le morphisme $F_{ij\ell}\to \Lambda_{ij\ell}$ (la source et le but étant munis
de leurs structures réduites) est quasi-lisse et purement de dimension relative $d$, et $F_{ij\ell}$
est purement de dimension $n$ car c'est l'image de $Y_{ij\ell}$ par un morphisme fini, et
car $Y_{ij\ell}$ est lui-même purement de dimension $n$, étant quasi-étale sur $Y$ ; il s'ensuit
que $\Lambda_{ij\ell}$ est purement de dimension $n-d$. 
Posons $X'=\coprod X_{ij\ell}, \Lambda=\coprod \Lambda_{ij\ell}$
et $Y'=\coprod Y_{ij\ell}$. Tous ces espaces sont compacts et 
$\Gamma$-stricts, $X'$ est muni d'un morphisme quasi-étale
vers $X$, l'espace $Y'$ est un domaine analytique de $Y\times_X X'$
se surjectant sur $Y$, et $\Lambda$ est un fermé de Zariski de $X'$
purement de dimension $n-d$ ; de plus lorsqu'on munit $\Lambda$ de sa structure
réduite le morphisme $Y'\to X'$ se factorise par $\Lambda$ (et $Y'\to \Lambda$ 
est surjective, mais nous ne nous en servirons pas). 
\textit{Il suffit de démontrer le théorème pour le morphisme $Y'\to X'$ et partant, en vertu
de la première situation particulière
étudiée plus haut, pour le morphisme $Y'\to \Lambda$}. 
Supposons en effet qu'il existe un
$\Gamma$-enjoliveur $(Z,S)$ de $X'$ de classe $\leq d+1$
et un domaine analytique compact et $\Gamma$-strict $V$ tels que les propriétés
suivantes soient satisfaites : 

\begin{itemize}[label=$\diamond$] 
\item La flèche $V\oslash_Z S\to Y'$
est surjective. 

\item La flèche $V\oslash_Z S\to Z$ se factorise par un morphisme
plat et surjectif sur un sous-espace analytique fermé réduit de $Z$
purement de dimension $n-d$. 

\item L'image de $S$
sur $X'$ est contenue dans celle 
de $Y'$ et de dimension
strictement inférieure à $n-d$. 
\end{itemize}
Choisissons une présentation
$\Gamma$-admissible
\[((Z=Z_{2d+2}\to Z_{2d+1}\to \cdots \to Z_1\to Z_0=X'), (S_i)_{0\leq i\leq 2d+2})\]
de $(Z,S)$ sur $X'$ (avec $S_0=\emptyset$ et $S_{2d+2}=S$). En remplaçant $Z_0$ par $X$
et la flèche quasi-étale $Z_1\to X'$ par la composée $Z_1\to X'\to X$, on obtient une présentation
admissible de $(Z,S)$ sur $X$ de longueur $d+1$, et $V$ peut être vu comme un domaine 
$k$-analytique compact et $\Gamma$-strict de $Y\times_X Z$. Il est alors immédiat que 
le
$\Gamma$-enjoliveur $(Z,S)$ de $X$ et le domaine $V$ satisfont les conclusions du théorème. 
Il suffit donc bien de démontrer le théorème pour la flèche $Y'\to \Lambda$, ce qui permet de
se ramener au cas où $X$
est purement de dimension $n-d$ (la flèche $Y'\to \Lambda$ était également relativement équidimensionnelle, 
mais cela n'a plus d'intérêt à ce stade de la démonstration donc nous oublions cette hypothèse).
L'image
de $\mathsf A(Y)$ est alors contenue
dans $\mathsf A(X)$, l'ouvert
de Zariski $\pl YX$ est dense
dans $Y$, et il
contient même
l'intégralité de la fibre
de tout point de $\mathsf A(Y)$
(lemme \ref{complements-dim} (2) et (3)). 
Appliquons
le théorème
\ref{theo-aplatissement-bis}
en prenant le triplet
$(\mathscr F,d,n)$
de son énoncé égal à 
$(\mathscr O_X, D,0)$.
Il assure l'existence d'un 
$\Gamma$-enjoliveur $(Z,S)$
de classe
$(D+1)$ de $X$ et d'un domaine analytique compact
et $\Gamma$-strict $V$ de $Y\times_X Z$
tel que les propriétés suivantes
soient satisfaites (en se rappelant 
que «plat en dimensions $\geq 0$»
signifie simplement «plat») : 

\begin{enumerate}[1]

\item $V\oslash_Z S \to Z$ est plat.  
\item L'image de
$V\oslash_Z S\to Y$ contient 
$Y\setminus \ql YX^{\mathrm{sat}}$, c'est-à-dire la 
réunion des fibres de $Y\to X$ entièrement
contenues dans $\pl YX$. 
\item L'image de $V\times_Z S$ sur $Y$ est contenue dans $\ql YX^{\mathrm{sat}}$, ce qui veut dire qu'elle
ne rencontre aucune fibre de $Y\to X$
entièrement contenue dans $\pl YX$. 
\item L'image de $S$
sur $X$ est contenue dans celle 
de $Y$ et est
de dimension $\leq n$ (notons que cette dernière
égalité n'apporte en fait aucune information, car elle était \textit{a priori}
évidente puisque $\dim Y=n$). 
\end{enumerate}
Puisque $V\oslash_Z S\to Z$ est plat, son image $Z'$ sur $Z$ est un
domaine analytique compact et $\Gamma$-strict de $Z$, et l'on peut remplacer
$Z$ par $Z'$ (et $S$ par $S\cap Z'$) sans modifier les propriétés
ci-dessus, ce qui permet de supposer
que la flèche  $V\oslash_Z S\to Z$
est surjective. 
L'image de $V\oslash_Z S\to Y$ est
une partie compacte de $Y$ contenant
la réunion des fibres
entièrement incluses dans $\pl YX$ ; 
elle contient en particulier $\mathsf A(Y)$, lequel est dense dans $Y$. 
Par conséquent, 
$V\oslash_Z S\to Y$ est surjective. 
Soit  $s$ un point de $S$
et soit $x$ son image
sur $X$. Pour terminer la démonstration, 
il reste à s'assurer que $d_k(x)\leq n-d-1$. 
Comme 
$V\oslash_Z S\to Z$ est surjectif, $s$ possède un antécédent $v$ sur $V$,
dont l'image $y$ sur $Y$ est un antécédent de $x$. 
Puisque $y$ est l'image d'un élément de $V\times_X S$, 
sa fibre n'est pas entièrement contenue dans $\pl YX$ ; 
ceci exclut que $x$ appartienne à $\mathsf A(X)$
(lemme \ref{complements-dim} (3)).
Par conséquent, 
$d_k(x)$ est strictement inférieur à $m=n-d$. 
\end{proof}

\section{Un théorème de Chevalley non archimédien}

Nous nous proposons dans cette section d'élucider autant qu'il est possible la structure de l'image d'un morphisme
entre espaces analytiques compacts et $\Gamma$-stricts. 
Un passage en fin de preuve requerra l'existence d'une borne uniforme
pour le cardinal des fibres géométriques d'un certain morphisme quasi-étale, 
aussi allons-nous commencer par une proposition qui établit
l'existence d'une telle borne, 
mais dans un contexte bien plus général que celui dont nous aurons besoin ; cette proposition nous semble avoir un intérêt propre, et fournit un premier
exemple d'énoncé qui se démontre en enjolivant un morphisme. 

\begin{prop}\label{taille-fibres-bornee}
Soit $Y\to X$ un morphisme d'espaces $k$-analytiques 
et soit $\mathscr F$ un faisceau cohérent sur $Y$. 
Supposons $X$ et $Y$ compacts. 
Pour tout $x\in X$ la quantité 
\[\lambda_x(\mathscr F):=\sum_{y\in Y_x, \dim_y Y_x=0}\dim_{\hr x} \mathscr F_{Y_x, y}\]
est finie, et elle est bornée indépendamment de $x$. 
\end{prop} 

\begin{rema}
La proposition analogue pour un morphisme de type fini entre schémas noethériens est vraie, et se démontre facilement
à l'aide du \textit{Main Theorem} de Zariski. Celui-ci possède une déclinaison analytique \cite[th. 3.2]{ducros2007}, 
mais on ne peut pas s'en servir directement ici, car le lieu des points en lesquels $Y$ est quasi-fini sur $X$
est un ouvert de $Y$ qui n'est en général pas compact (alors que dans le monde schématique, il est quasi-compact). 
\end{rema}

\begin{proof}[Démonstration de la proposition \ref{taille-fibres-bornee}]
On procède en plusieurs étapes. Notons qu'en raison de la compacité des
espaces en jeu, l'assertion est G-locale sur $X$ comme sur $Y$.

On suppose tout d'abord que $Y\to X$ est quasi-fini, c'est-à-dire
purement de dimension nulle. 
En raisonnant G-localement à la source et au but
on peut supposer $X$ et $Y$ affinoïdes puis, en raison de la version analytique  \textit{Main Theorem}
de Zariski \cite[th. 3.2]{ducros2007}, que $Y\to X$ admet
une factorisation $Y\to T\to X$ où $T$ est affinoïde, où $\pi \colon Y\to T$ est fini, et où $T\to X$ est quasi-étale. 
On a pour tout $x$
appartenant à $X$ l'égalité $\lambda_x(\mathscr F)=\lambda_x(\pi_*
\mathscr F)$, et il suffit donc de montrer
que cette dernière quantité est bornée indépendamment de $x$. En considérant une surjection 
de $\mathscr O(T)^N$ sur $\mathscr F(Y)$ (pour $N$ convenable), on voit qu'il suffit de montrer que
$\lambda_x(\mathscr O_T)$ est borné
indépendamment de $x$. Comme c'est une propriété locale sur $T$, on peut supposer que $T$ est un domaine affinoïde
d'un espace affinoïde $S$ fini et étale sur un domaine affinoïde $X'$ de $X$, et il suffit de trouver une borne 
pour $\lambda_x(\mathscr O_S)$ indépendante du point $x\in X'$ ; on peut alors
prendre le cardinal de n'importe quelle famille génératrice
finie du $\mathscr O(X')$-module de type fini $\mathscr O(S)$. Ceci termine
la preuve dans le cas quasi-fini.

Passons au cas général. 
L'assertion à démontrer ne mettant en jeu que les fibres de $Y\to X$, on peut supposer
$X$ réduit, et également irréductible en raisonnant composante par composante. 
On procède alors par récurrence sur la dimension $m$ de $X$. 
Si $m=0$ l'espace $X$ ne compte qu'un point, et $Y\to X$ a donc une seule fibre, qui est compacte
et ne compte par conséquent qu'un nombre fini de points isolés, et \textit{a fortiori} qu'un nombre fini
de points rigides isolés ; la proposition est alors évidente. 
On suppose désormais $m>0$ et le résultat vrai en dimension $m-1$. 
En raisonnant G-localement sur $Y$,  on peut le supposer affinoïde ;
soit $B$ l'anneau de ses fonctions analytiques.  Le faisceau $\mathscr F$
est donné par un $B$-module de type fini $M$, qui par noethérianité de $B$ admet
une filtration finie dont les quotients successifs sont de la forme $B/\mathfrak p_1,\ldots, B/\mathfrak p_n$ où les
$\mathfrak p_i$ sont des idéaux premiers. On peut donc supposer que $M$ est lui-même de la forme
$B/\mathfrak p$ avec $\mathfrak p$ premier. Si $Z$ désigne le sous-espace analytique fermé
$\mathscr M(B/\mathfrak p)$ alors pour tout $y\in Z$ on a 
$\dim_y {f|Z}\leq \dim_y f$, si bien qu'il suffit de majorer les sommes étudiées en remplaçant $Y$
par $Z$. Autrement dit, on s'est ramené au cas où $Y$ est intègre et où $\mathscr F=\mathscr O_Y$. 
Si la dimension générique de $Y\to X$ est strictement positive les sommes étudiées sont toutes nulles ; 
on peut donc supposer que la dimension générique de $Y\to X$ est nulle. 
Le théorème \ref{theo-embellissement} assure l'existence d'un enjoliveur $(Z,S)$ sur $X$ et d'un domaine analytique
compact $V$ de $Y\times_X Z$ tel que $V\oslash_Z S\to Y$ soit surjective et telle que $V\oslash_Z S\to Z$ se factorise
par une surjection plate et quasi-finie $Z\to F$ où $F$ est un sous-espace analytique fermé réduit de $Z$ purement 
de dimension $\dim Y$. 
Notons $g$ le morphisme $Z\to X$. 
En considérant une décomposition admissible 
\[(Z_{2r}\to Z_{2r-1}\to \cdots \to Z_0, (S_i))\] de $(Z,S)$ et en appliquant 
le cas quasi-fini déjà traité aux morphismes quasi-étales 
$Z_{2i+1}\to Z_{2i}$ pour $i$ compris entre $0$
et $n-1$, on obtient 
l'existence d'un entier $d$
tel que pour tout $x\in X\setminus g(S)$, \[\sum_{z\in Z,g(z)=x}
[\hr z:\hr x]\leq d\]
(notons que comme $x\notin g(S)$, la fibre $g^{-1}(x)$
est contenue dans le lieu quasi-étale de $Z$ sur $X$, et on sait
donc \textit{a priori} qu'elle consiste en un nombre fini de points $z$
tels que $\hr z$ soit fini séparable sur $\hr x$). 
Le diviseur de Cartier $S$ de $Z$ est purement de dimension $m-1$. 
En appliquant l'hypothèse de récurrence
à chacune de ses composantes irréductibles réduites on obtient l'existence d'un 
entier $N_1$ tel que $\lambda_s (\mathscr O_{Y\times_XZ})\leq N_1$ pour tout $s\in S$. 
Si $x$ est un point de $X$ possédant un antécédent $s$ sur $S$
on a donc 
\[\lambda_x(\mathscr O_Y)=\lambda_s(\mathscr O_{Y\times_X Z})\leq N_1.\]
La flèche $V\oslash_Z S\to Z$ est quasi-finie, et puisque la proposition a été établie dans ce cas, 
il existe $N_2$
tel que $\lambda_z(\mathscr O_{V\oslash_Z S})\leq N_2$ pour tout $z$
dans $Z$. 
Soit $x$ un point de $X$ n'ayant 
aucun antécédent sur $S$. 
La fibre de $Z$ en $x$ 
consiste comme
on l'a vu plus haut en un ensemble fini de points
$z_1,\ldots, z_r$ tels que $\hr {z_i}$ soit fini séparable sur
$\hr x$ pour tout $i$ et tels que $\sum_i [\hr{z_i}:
\hr x]\leq d$. En particulier, $r$ est inférieure ou
égal à $d$, si bien que
$\sum_i \lambda_{z_i}(\mathscr O_{V\oslash_Z S})\leq dN_2$.
Soit $y$ un point de $Y_x$ en lequel
$Y_x$ est de dimension nulle. 
Par surjectivité de $V\oslash_ZS\to Y$, il existe un antécédent $v$ de $y$ sur
$V\oslash_Z S$. Le point $v$ est situé au-dessus de l'un des $z_i$, 
et n'appartient donc pas à $S$ ; il en résulte que 
$\mathscr O_{(V\oslash_Z S)_{z_i},v}=\mathscr O_{V_{z_i}, v}$, et ce dernier coïncide avec 
$\mathscr O_{(Y\times_X Z)_{z_i},v}$ car $(Y\times_X Z)_{z_i}$ est de dimension 
nulle en $v$. Par conséquent, $\mathscr O_{(V\oslash_Z S)_{z_i}, v}$
est un sommande de $\mathscr O_{Y_x,y}\otimes_{\hr x}\hr{z_i}$. 
Si $E$ désigne la fermeture séparable de $\hr x$ dans l'anneau 
$\mathscr O_{Y_x,y}$ il existe donc un composé $F$ de $\hr{z_i}$ et $E$ sur $\hr x$ tel que
$\mathscr O_{V\oslash_Z S)_{z_i}, v}$ soit égale à 
$\mathscr O_{Y_x,y}\otimes_EF$. 
Il vient 

\begin{eqnarray*}
\dim_{\hr {z_i}}\mathscr O_{
(V\oslash_ZS)_{z_i},v}&=&(\dim_E \mathscr O_{Y_x,y})[F:\hr{z_i}]\\
&=&(\dim_{\hr x}\mathscr O_{Y_x,y})\frac{[F:\hr{z_i}]}{[E:\hr x]}.
\end{eqnarray*}
Comme on a par ailleurs 
\[\frac{[F:\hr{z_i}]}{[E:\hr x]}=
\frac{[F:\hr x]\cdot[F:E]}
{[F:\hr x]\cdot[\hr{z_i}:\hr x]}
\geq \frac 1{[\hr{z_i}:\hr x]}\geq \frac 1{d}\]
on voit que finalement 
$\dim_{\hr {z_i}}\mathscr O_{
(V\oslash_ZS)_{z_i},v}\geq \frac{\dim_{\hr x}\mathscr O_{Y_x,y}}d.$
Ceci vaut
pour tout point 
$y$ de la fibre $Y_x$ en laquelle celle-ci est de dimension nulle. 
Par conséquent, 
$\sum_i \lambda_{z_i}(\mathscr O_{V\oslash_Z S})
\geq 
\frac{\lambda_x(\mathscr O_Y)}d$, et puisqu'on a vu plus haut que
$\sum_i \lambda_{z_i}(\mathscr O_{V\oslash_Z S})\leq dN_2$, on obtient
la majoration
$\lambda_x(\mathscr O_Y)\leq d^2N_2$. 

\textit{Conclusion.}
Soit $x$ un point de $X$. On a vu que si $x$
appartient à l'image de $S$
alors $\lambda_x(\mathscr O_Y)\leq N_1$
et que $\lambda_x(\mathscr O_Y)\leq 
d^2N_2$ dans le cas contraire. 
On a donc 
$\lambda_x(\mathscr O_Y)\leq\max(N_1,d^2N_2)$, ce qui achève la
démonstration.
 \end{proof}

%\section{Un théorème de Chevalley non archimédien}

Nous pouvons maintenant aborder l'étude
des images de morphismes entre espaces $k$-analytiques
compacts et $\Gamma$-stricts. Nous allons 
commencer par le cas le plus simple, auquel le cas général va 
se ramener ; il concerne la composée d'un morphisme quasi-étale et d'une immersion fermée.

\begin{theo}\label{theo-casqe}
Soit $f\colon Y\to X$ un morphisme compact et quasi-étale
entre espaces $k$-analytiques $\Gamma$-stricts, et
soit $Z$ un fermé de Zariski de
$Y$. Pour tout entier $d$, notons $X_d$ l'ensemble des points 
de $X$ dont la fibre géométrique sous $f$ est au moins
de cardinal $d$.

\begin{enumerate}[1]
\item Pour tout $d$, l'ensemble $X_d$ est un domaine analytique
$\Gamma$-strict et fermé
de $X$. 

\item Pour tout $d$, l'intersection $f(Z)\cap(X_d\setminus X_{d+1})$
est un fermé de Zariski de $X_d\setminus X_{d+1}$.

\item Pour tout $x\in f(Z)$ et tout domaine analytique 
$V$ de $X$
contenant $x$ on a les égalités
\[\dim_x (f(Z)\cap V)=\dim_x f(Z)=\max_{z\in Z,f(z)=x} \dim_z Z.\]

\end{enumerate}

\end{theo}

\begin{proof}
Montrons tout d'abord (1) et (2). 
Ces énoncés sont G-locaux sur $X$, ce qui permet de supposer $X$ et $Y$
compacts. Le cardinal des fibres
géométriques de $f$ est alors uniformément borné par un certain entier 
$n$, et l'on raisonne par récurrence sur $n$. 
Si $n=0$ alors $Y=\emptyset$ et $Z=\emptyset$, et l'on a
$X_d=X$ si $d=0$ et $X_d=\emptyset$ sinon, et le
théorème est évident. 
Supposons maintenant $n>0$ et le théorème vrai pour les entiers $<n$. 
Notons $p$ et $q$ les deux projections de $Y\times_X Y$ sur $Y$ ; 
nous considérons $p\colon Y\times_X Y\to Y$ comme le changement de base
de $f$ \textit{via} lui-même. Notons $T$ le saturé de $Z$ sous
la relation d'équivalence induite par $f$ ; on a les égalités
$f(T)=f(Z)$ et $T=p(q^{-1}(Z))$. Pour tout $d$, on pose
$Y_d=f^{-1}(X_d)$ ; c'est exactement l'ensemble des points de $Y$
dont la fibre géométrique sous $p$ est de cardinal au moins égal à $d$. 
Puisque $f$ est quasi-étale, la diagonale $\delta \colon Y
\to Y\times_X Y$ identifie $Y$ à un ouvert fermé de $Y\times_X Y$. 
Autrement dit l'on dispose
d'un isomorphisme $Y\times_X Y\simeq Y\coprod
Y'$ modulo lequel $p|_Y$ et $q|_Y$ sont toutes
deux égales à $\mathrm{Id}_Y$. Si l'on pose $Z'=q^{-1}(Z)\cap Y'$
alors $T=Z\cup p(Z')$. 
Par construction $Y_0=Y_1=Y$
et 
$Y_d$ est pour tout $d\geq 1$ l'ensemble des points de $Y$ en lesquels
la fibre géométrique de $p|_{Y'}$ est de cardinal au moins
$d-1$. Par ailleurs, le morphisme
$p|_{Y'}\colon Y'\to Y$ est quasi-étale et a des fibres géométriques
de cardinal inférieur ou égal à $n-1$. 
On en déduit que pour tout $d$, le sous-ensemble $Y_d$ de $Y$ en est un domaine
analytique compact et $\Gamma$-strict (par hypothèse de récurrence si
$d\geq 1$, et directement si $d=0$), et que l'intersection $p(Z')\cap (Y_d\setminus Y_{d+1})$
est un fermé de Zariski de $Y_d\setminus Y_{d+1}$
(par hypothèse de récurrence si $d\geq 1$, et directement
si $d=0$). 
Il s'ensuit comme $f$ est quasi-étale et compact 
que $X_d=f(Y_d)$ est pour tout $d$ un domaine analytique 
compact et $\Gamma$-strict de $X$ (cela vaudrait en général si $f$
était simplement supposé plat, mais le cas particulier de la dimension
relative nulle est beaucoup plus simple,
\textit{cf.} \cite[prop. 9.1.1]{ducros2018}), 
ce qui montre (1).

L'intersection
$f(Z)\cap (X_0\setminus X_1)$ est vide
puisque $X_1=f(Y)$, et c'est en particulier
un fermé de Zariski
de $X_0\setminus X_1$. 
Soit maintenant $d$ un entier 
supérieur ou égal à $1$.
Posons $V=Y_d\setminus Y_{d+1}$
et $U=f(V)$.
Comme $X_d$ et $X_{d+1}$
sont contenus dans $X_1= f(Y)$, 
on a $U=X_d\setminus X_{d+1}$
et $V$ est égal à $f^{-1}(U)$.
Puisque $T$ est égal à $Z\cup p(Z')$
et que
$p(Z')\cap V$ est un fermé de Zariski de $V$, 
l'intersection
$T\cap V$
est 
un fermé de Zariski de $V$.
Le faisceau cohérent d'idéaux
$\mathscr J$ de $V$ définissant 
la structure réduite de $T\cap V$ est muni
de données de descente
naturelles relative au morphisme
$V\to U$ 
(cela provient du fait que $T\cap V=f^{-1}(Z\cap U)$
ensemblistement, et
que les morphismes quasi-étales préservent le caractère réduit). 
Par descente des faisceaux cohérents sous les morphismes
compacts et fidèlement plats \cite[th. 3.3]{ducros2021a}, 
le faisceau $\mathscr J$ provient d'un faisceau
cohérent d'idéaux $\mathscr I$ sur $U$.
Il s'ensuit que
$f(Z)\cap U$ est un fermé de Zariski de $U$ (de structure réduite définie
par $\mathscr I$), ce qui montre (2).

Montrons maintenant (3). 
Soit $x\in f(Z)$,
et soit $V$ un domaine analytique de $X$
contenant $x$. Posons
$W=f^{-1}(V)$ ; on a alors
$f(Z)\cap V=f(Z\cap W)$. 
Soient $z_1,\ldots, z_r$ les antécédents de $x$ sur
$Z$ et soit $U$ un voisinage ouvert de $x$ dans $f(Z)$, 
suffisamment petit pour que l'image
réciproque $U'$ de $U$ sur $Z$ s'écrive comme
une union disjointe $\coprod U'_i$ où $U'_i$
est pour tout $i$ un voisinage de $z_i$
de dimension $\dim_{z_i}Z$
(c'est possible par propreté topologique). 
Le morphisme
$f$ étant de dimension relative nulle, 
on a 
\begin{eqnarray*}
\dim U&=&\sup_{t\in U}d_k(t)\\
&=&
\sup_{z\in U'}d_k(f(z))\\
&=&\sup_{z\in U'}
d_k(z)\\
&=&\dim U'\\
&=&\max_i \dim U'_i\\
&=&\max_i\dim_{z_i}
Z\end{eqnarray*}

et 
\begin{eqnarray*}
\dim (U\cap V)&=& \sup_{t\in U\cap V}d_k(t)\\
&=&
\sup_{z\in U'\cap W}d_k(f(z))\\
&=&\sup_{z\in U'\cap W}
d_k(z)\\
&=&\dim (U'\cap W)
\\
&=&\max_i \dim (U'_i\cap W)\\
&=&\max_i\dim_{z_i}
Z.
\end{eqnarray*}
Ceci vaut pour tout $U$ suffisamment petit. En conséquence on a 
\[\dim_x f(Z)=\max _i\dim_{z_i}Z\;\text{et}\;
\dim_x (f(Z)\cap V)=
\max _i\dim_{z_i}Z.\qedhere\]
\end{proof}

\begin{defi}\label{def-cellule}
Soit
$X$ un espace $k$-analytique
et $\Gamma$-strict et soit $C$
une partie localement fermée de $X$ ; on note
$\partial C$ le complémentaire de
$C$ dans son adhérence $\overline C$ (dont $C$ est un ouvert). 
Soit $n$ un entier. Nous allons définir
récursivement sur $n$ le fait pour $C$
d'être une \textit{$n$-cellule}
$\Gamma$-stricte.  

\begin{itemize}[label=$\diamond$]
\item $C$ est une $0$-cellule
$\Gamma$-stricte si et seulement si $C$
est un fermé de Zariski de $X$ non vide et de dimension nulle. 

\item Si $n>0$ alors $C$ est une $n$-cellule
$\Gamma$-stricte si et seulement si
les deux conditions suivantes sont satisfaites : 
\begin{itemize}[label=$\bullet$] 
\item il existe un espace $k$-analytique
$\Gamma$-strict $Y$, un morphisme quasi-étale compact 
$f\colon Y \to X\setminus \partial C$ dont le cardinal 
des fibres géométriques est borné, et un fermé de Zariski
$Z$ non vide de $Y$ purement de dimension $n$ tel que
$C=f(Z)$ ; 
\item le bord $\partial C$ s'écrit
comme une union finie $\bigcup \overline{C_i}$ où chaque $C_i$ est
une $n_i$-cellule
$\Gamma$-stricte pour un certain entier $n_i<d$.
\end{itemize}

\end{itemize}
\end{defi}

\begin{enonce}[remark]{Commentaires}
Soit $C$ une $n$-cellule $\Gamma$-stricte et soient
$Y, f$ et $Z$ 
comme dans la définition ci-dessus. Il résulte 
du théorème \ref{theo-casqe}, et du fait que le cardinal
des fibres géométriques de $f$ est borné, qu'il existe une suite 
décroissante $(X_d)_{d\geq 0}$ de domaines analytiques 
fermés et $\Gamma$-stricts de $X\setminus \partial C$ 
telle que $X_0=X\setminus \partial C$ et $X_d=\emptyset$ pour
$d$ assez grand
et telle que $C\cap(X_d\setminus X_{d+1})$ soit 
pour tout $d$ un fermé de Zariski de $X_d\setminus X_{d+1}$ ; 
et il résulte également de \textit{loc. cit.}
que $\dim_x C=n$ pour tout $x\in C$. L'entier $n$ est donc uniquement
déterminé par $C$ : c'est sa dimension, et c'est aussi la dimension
de $\overline C$, et même la dimension locale
$\dim_x \overline C$ en tout $x$ de $\overline C$. 
\end{enonce}

\begin{rema}
Si $D$ est un fermé de $X$ s'écrivant $\overline C$ pour une certaine $n$-cellule
$\Gamma$-stricte $C$, l'entier $n$
est uniquement déterminé (on a vu que c'est la dimension de $D$), 
mais la cellule $C$ ne l'est pas en général.  En effet, supposons par exemple que $X$ est le disque unité fermé. Alors $X$
peut lui-même être vu comme une $1$-cellule
$\Gamma$-stricte de bord vide, mais $C:=X\setminus \{0\}$
est une autre $1$-cellule $\Gamma$-stricte,
de bord cette fois-ci égal à $\{0\}$, dont l'adhérence est égale à $X$.

Peut-on espérer en général que si $D$ est une partie de $X$ égale à l'adhérence d'une $n$-cellule
$\Gamma$-stricte il existe 
une plus grande $n$-cellule $\Gamma$-stricte $C$ telle que $D=\overline C$ ? Nous l'ignorons.

\end{rema}

\begin{exem}
Posons 
\[X=\mathscr M(k\{T_1,T_2\})\;\text{et}\;
Y=\mathscr M(k\{T_1,T_2,T_3\}/(T_3T_2-T_1)),\]
et soit
$f\colon Y\to X$ la flèche évidente. 
Un calcul direct montre que $f(Y)$ est le fermé
de $X$ défini par l'inégalité $\abs{T_1}
\leq \abs{T_2}$ ; soit $C$ son intersection avec
le complémentaire $U$ de l'origine dans $X$. 
On a alors $f(Y)=\overline C$ et $\partial C=\{(0,0)\}$. 
De plus, soit $x$ un point de $U$. 
Si $\abs{T_2(x)}=0$ alors $x$ n'est pas situé sur $C$, et possède 
donc un voisinage ouvert dans $U$ qui ne rencontre pas $C$. 
Et si $\abs{T_2(x)}\neq 0$, le lieu d'inversibilité $U'$ de $T_2$
sur $U$ est un voisinage ouvert de $x$ dans $U$, et $C\cap U'$ est le domaine
analytique fermé de $U'$ défini par l'inégalité $\abs{T_1/T_2}\leq 1$. 
Comme être un domaine analytique fermé est une propriété locale, on voit que 
$C$ est un domaine analytique fermé de $U$. Par conséquent, $C$ est une $2$-cellule. 
\end{exem}

Nous pouvons maintenant énoncer notre avatar non archimédien du théorème de
Chevalley, dont l'exemple ci-dessus sera \textit{a posteriori} une illustration
très simple.

\begin{theo}\label{theo-chevalley-na}
Soit $f\colon Y\to X$ un morphisme entre espaces
$k$-analytiques compacts
et $\Gamma$-stricts. On
suppose que $Y$ est non vide et équidimensionnel,
et l'on note $n$ sa
dimension ; 
on suppose aussi
qu'il existe un entier 
$d$ tel que $f$
soit génériquement de dimension
$d$. 

\begin{enumerate}[1]
\item 
Il existe une $(n-d)$-cellule
$\Gamma$-stricte $C$ de $X$ tel que 
$f(Y)=\overline C$. 
\item On a $f(\mathsf A(Y))\subset C$. 

\end{enumerate}

\end{theo}

\begin{proof}
Remarquons pour commencer 
que $d_k(f(y))=n-d$
pour tout $y\in \mathsf A(Y)$
(lemme \ref{complements-dim} (1)) ; 
par conséquent, si (1) est vraie et si
$y$ est un point de $Y$, on ne
peut avoir $f(y)\in \partial C$, 
si bien que $f(y)\in C$, d'où (2). 

Il reste donc à démontrer (1), 
ce que l'on fait
par récurrence forte
sur la dimension $n-d$
de l'image $f(Y)$. Nous supposons
donc le résultat vrai en dimensions $<n-d$. 
L'énoncé du théorème est insensible
aux nilpotents, ce qui permet de
remplacer les espaces
en jeu par les espaces réduits
associés, c'est-à-dire
de supposer $X$ et $Y$
réduits.  
Le théorème \ref{theo-embellissement} assure
alors
l'existence d'un enjoliveur
$\Gamma$-strict $(Z,S)$ sur $X$
et d'un domaine
analytique compact et $\Gamma$-strict $V$
 de $Y\times_X Z$ 
tels que les propriétés suivantes soient satisfaites : 

\begin{enumerate}[a]

\item La flèche $V\oslash_Z S\to Y$
est surjective. 

\item La flèche $V\oslash_Z S\to Z$ se factorise par un morphisme
plat et surjectif sur un sous-espace analytique fermé réduit $F$ de $Z$
purement de dimension $n-d$. 

\item L'image de $S$
sur $X$ est contenue dans celle 
de $Y$ et de
dimension $<n-d$. 
\end{enumerate}
Il assure aussi que $S\cap F$ est d'intérieur vide dans $F$. 
Soit $g$ le morphisme
de $Z$ vers $X$. 
Puisque $V\oslash_S Z\to Y$ est surjective, $f(Y)$ est égal à l'image de  $V\oslash_Z S\to X$, 
c'est-à-dire encore à $g(F)$. Posons $\Sigma=g^{-1}(g(S))$ ; c'est un fermé de $Z$ et 
l'on a 
\[f(Y)=g(F)=g(F\setminus \Sigma)\coprod g(F\cap \Sigma)=g(F\setminus \Sigma)\coprod g(S),\]
où la dernière égalité provient du fait que $g(S)\subset f(Y)$. 
Si $z$ est un point de $\mathsf A(F)$ il n'est pas situé sur $S$, si bien que $Z\to X$ est quasi-étale en $z$ ; 
en particulier $d_k(g(z))=d_k(z)=n-d$, ce qui à l'aide de (c) assure que $g(z)\notin g(S)$ ; par conséquent, 
$\Sigma$ ne rencontre pas $\mathsf A(F)$. Le morphisme $Z\setminus \Sigma \to X\setminus g(S)$ se déduit de $Z\to X$ par
changement de base le long de l'immersion ouverte $(X\setminus g(S))\hookrightarrow X$ ;
par conséquent, il est compact, et il est quasi-étale puisque $S\subset \Sigma$, et le cardinal 
de ses fibres géométriques est uniformément borné 
(cela découle de la proposition \ref{taille-fibres-bornee} mais en est en fait un cas
particulier
bien plus facile établi au cours de sa preuve). 
Posons $C=g(F\setminus\Sigma)$. 
Si $z\in \mathsf A(F)$ on a vu plus haut que $z$ n'appartient pas à $\Sigma$, et $g(z)$ appartient donc à $C$. Ceci entraîne, 
$\mathsf A(F)$ étant dense dans $F$, que $C$ est dense dans $f(Y)=g(F)$.
Il suffit dès lors pour conclure de démontrer que $C$
est une $(n-d)$-cellule ; au vu de sa définition, le seul point qui reste à 
vérifier est
qu'on peut écrire $\partial C$ comme 
une union finie d'adhérences de
cellules $\Gamma$-strictes
de dimension strictement inférieure à 
$n-d$.

Puisque $C$ est dense dans $g(F)$, 
on a $\partial C=g(S)$. 
Pour tout couple
$e=(e_1,e_2)$ d'entiers
avec $e_2\leq e_1$ notons
$S_e$
la réunion des composantes
irréductibles de $S$ de dimension 
$e_1$ dont la
dimension générique relative sur $X$ est égale à $e_2$. 
Si $e$ est tel que $S_e$ soit non vide on
a d'après (c) l'inégalité
$\dim g(S_e)<n-d$, 
et notre hypothèse de récurrence
assure alors que $g(S_e)$ est de la forme $\overline D_e$ où $D_e$ est une 
cellule
$\Gamma$-stricte
de dimension $\dim g(S_e)$. 
Si $E$ désigne l'ensemble (fini) des
indices $e$ tels que $S_e\neq \emptyset$ on a
$S=\bigcup_{e\in E}S_e$. Il vient 
$\partial C=\bigcup_{e\in E}\overline {D_e}$.
\end{proof}

\begin{rema}
Le théorème \ref{theo-chevalley-na}
ci-dessus décrit (partiellement) la structure de l'image d'un morphisme $f$
entre espaces analytiques compacts et $\Gamma$-stricts
sous des hypothèses minimales (on demande simplement que la source
de $f$ soit équidimensionnelle
et non vide et que $f$ ait une dimension générique). 
Mais en dévissant sa preuve on s'aperçoit que le cas particulier où $f$ est plat, déjà connu, est utilisé
de manière cruciale. En effet, dans l'assertion (b) du début de preuve, 
c'est-à-dire dans l'assertion (A2) du théorème \ref{theo-embellissement}, la surjectivité
découle \emph{in fine}
du fait qu'un morphisme plat entre espaces compacts et $\Gamma$-stricts
a pour image un domaine analytique $\Gamma$-strict du but
(voir le traitement du cas génériquement plat lors de la preuve du théorème
\ref{theo-embellissement}). 

La situation est donc renversée par rapport à celle du monde schématique : dans ce dernier
on démontre en effet que l'image d'un morphisme plat et de présentation finie est ouverte en utilisant 
le fait que cette image est stable par générisation (ici intervient la platitude) et qu'elle est constructible 
\emph{par le théorème de Chevalley}. On voit donc qu'en géométrie algébrique chercher à montrer (à l'instar de ce que nous faisons ici) 
le théorème de Chevalley par une réduction au cas plat \emph{via} les techniques de Raynaud et Gruson conduirait à un raisonnement 
circulaire. 
\end{rema}

\begin{rema}
Lorsque $\Gamma=\{1\}$, la description de l'image d'un morphisme plat entre espaces analytiques compacts 
et $\Gamma$-stricts est due à Raynaud, \emph{cf.} \cite[cor. 5.11]{frg2}, avec une
preuve reposant sur des techniques de géométrie formelle.
Le théorème 
\ref{theo-chevalley-na} pourrait-il lui aussi, lorsque $\Gamma=\{1\}$, être démontré directement (sans
passer par nos théorèmes d'aplatissement analytiques) par des méthodes formelles ? 
Cela nous semble assez peu probable, pour la raison suivante : l'un des points clef de l'approche de Raynaud est que dans le cas d'un morphisme
\emph{plat} $\mathfrak Y\to \mathfrak X$ entre schémas formels, l'image du morphisme analytique induit $\mathfrak Y_\eta\to \mathfrak X_\eta$ 
coïncide avec le tube de l'image du morphisme schématique induit $\mathfrak Y_{\mathrm s}\to \mathfrak X_{\mathrm s}$
(et comme cette dernière image est un ouvert de Zariski de $\mathfrak X_{\mathrm s}$, son tube est un
domaine strictement analytique compact de $\mathfrak X_\eta$). 
Mais ceci est grossièrement faux en général
sans hypothèse de platitude, et il n'y a donc à notre connaissance aucun moyen de construire
un modèle formel $\mathfrak Y\to \mathfrak X$ d'un morphisme quelconque $Y\to X$ entre espaces strictement $k$-analytiques compacts
tel que l'image de $Y\to X$ puisse être décrite en termes du morphisme de schémas $\mathfrak Y_{\mathrm s}\to \mathfrak X_{\mathrm s}$ . 

\end{rema}

\section{Dimension centrale et anneau local d'un filtre affinoïde}

Nous nous servirons dans ce qui suit de la théorie de la réduction 
des germes (ponctuels) d'espaces $k$-analytiques, et plus précisément
de sa variante $\Gamma$-stricte ; nous renvoyons le lecteur au chapitre 3 
de \cite{ducros2018} (et spécialement aux sections 3.4 et 3.5) pour les définitions, 
notations et propriétés de base. Cette théorie repose sur l'algèbre commutative graduée 
décrite dans \cite{temkin2004}, au sujet de laquelle on pourra aussi consulter
l'appendice A de \cite{ducros2018}. 

\begin{defi}\label{def-filtreaff}
Soit $X$ un espace $k$-analytique et soit $x$
un point de $X$. Un \textit{filtre affinoïde}
sur l'espace pointé $(X,x)$ 
est un ensemble $\Theta$ de domaines affinoïdes 
de $X$ contenant $x$ qui est filtrant et est tel que 
tout voisinage de $x$ dans $X$
contienne un élément de $\Theta$. 
\end{defi}

\begin{enonce}[remark]{Exemples}\label{exem-filtres}
En pratique, nous considérerons les filtres
affinoïdes suivantes sur $(X,x)$ : 

\begin{itemize}[label=$\diamond$] 
\item lorsque le germe $(X,x)$ est bon, le filtre formé
de tous les voisinages affinoïdes 
de $x$ dans $X$, que nous nous permettrons de noter encore $(X,x)$
 ; 
\item
lorsque le germe $(X,x)$
est $\Gamma$-strict, 
le filtre $(X\grot^\Gamma,x)$ 
formé de tous les domaines affinoïdes
$\Gamma$-stricts de $X$ contenant $x$ ;

\item
lorsque le germe $(X,x)$
est $\Gamma$-strict et que $\xi$
est un élément de la réduction
de Temkin $\Gamma$-graduée
$\widetilde{(X,x)}^\Gamma$,
le filtre $(X\grot^\Gamma,x,\xi)$ 
formé de tous les domaines affinoïdes
$\Gamma$-stricts $V$
de $X$ contenant $x$ et tels que
$\widetilde {(V,x)}^\Gamma$ contienne $\xi$. 
\end{itemize}
(Lorsque $\Gamma=\R_+^\times$,
nous omettrons
le plus souvent de l'indiquer en 
exposant). 
\end{enonce}

Soit $\Theta$
un filtre affinoïde sur $(X,x)$. Nous 
noterons
$\dimc \Theta$ la
\textit{dimension centrale}
de $\Theta$, c'est-à-dire
le minimum des dimensions
de $\adhp xV$ pour $V$ appartenant à 
$\Theta$.
On a de manière évidente
\[d_k(x)\leq \dimc \Theta\leq \dim \adhp xX\leq \dim_x X.\]
La dimension centrale $\dimc \Theta$ coïncide avec $\dim \adhp xX$ si et seulement si
$\dim \adhp xV=\dim \adhp xX$ pour tout domaine 
affinoïde $V\in \Theta$, ce qui revient à demander que 
$\adhp xV$ soit une composante irréductible
de $V\cap \adhp xX$ pour tout tel $V$. 
La dimension centrale $\dimc \Theta
$ coïncide avec $\dim_x X$ si et seulement si $\adhp xV$ est une composante irréductible
de $V$ pour tout $V\in\Theta$. Cette condition est notamment satisfaite dès que 
$d_k(x)=\dim_x X$, c'est-à-dire dès que $x\in \mathsf A(X)$. Mais elle l'est aussi
par exemple lorsque $\dim X=1$ et lorsque $x$ n'est pas un point rigide.

Nous noterons $\mathscr O_{\Theta}$ 
la colimite des $\mathscr O(V)$ pour $V$ appartenant à
$\Theta$. Lorsque $\Theta$ est de l'un des trois types 
décrits en \ref{exem-filtres}
nous le ferons figurer 
sans parenthèses
en indice et écrirons simplement
$\mathscr O_{X,x}, \mathscr O_{X\grot^\Gamma,x}$ et
$\mathscr O_{X\grot^\Gamma,x,\xi}$ (si $X$ est bon, la notation 
$\mathscr O_{X,x}$ de ce paragraphe a donc bien son sens habituel : elle désigne
l'anneau local en $x$ du faisceau $\mathscr O_X$ restreint à la catégorie des ouverts de $X$). 

Soit
$V$ un domaine affinoïde appartenant à $\Theta$ et soit $U$ un voisinage de $x$ dans $V$.
Par définition, $U$ contient une partie de la forme $U'\cap V$, où $U'$ est un voisinage ouvert de $x$
dans $X$ ; puisque $\Theta$ est un
filtre affinoïde de $(X,x)$, il existe un domaine
affinoïde $V'$ de $U'$ appartenant à $\Theta$, et un domaine affinoïde $W$ de $V\cap V'$ appartenant
à $\Theta$; en particulier, $W\subset U$. 
Il s'ensuit que la flèche naturelle $\mathscr O(V)\to \mathscr O_\Theta$ se factorise par
$\mathscr O_{V,x}$. Ceci valant pour tout $V\in \Theta$, on voit que $\mathscr O_\Theta$ peut
également se décrire comme la colimite de la famille des $\mathscr O_{V,x}$ pour $V\in \Theta$.

\begin{exem}\label{exemple-huber}
Si la valeur absolue de $k$ n'est pas triviale et si $X$ est strictement
$k$-analytique, se donner un élément $\xi\in \widetilde {(X,x)}^{\{1\}}$
revient à se donner une spécialisation $x^+$ de $x$
dans l'espace adique $X^{\mathrm{ad}}$ associé à $X$, et $\mathscr O_{X\grot^{\{1\}},x,\xi}$
est alors simplement l'anneau local de Huber $\mathscr O_{X^{\mathrm{ad}},x^+}$. 
\end{exem}

Soient $\Phi$ et $\Theta$
deux filtres 
affinoïdes
sur $(X,x)$. Nous dirons que
$\Phi$ raffine $\Theta$ et
écrirons
$\Phi\preceq \Theta$
si
tout élément
de $\Theta$
contient un élément de
$\Phi$. Nous dirons que $\Phi$ est
équivalent à $\Theta$ et écrirons 
$\Phi\sim \Theta$ si $\Phi
\preceq \Theta$ et $\Theta\preceq \Phi$.

Si 
$\Phi\preceq \Theta$
on a
un morphisme naturel
$\mathscr O_\Theta
\to \mathscr O_\Phi$, qui est 
un isomorphisme si
$\Phi\sim \Theta$. Et l'on 
 par ailleurs
$\dimc \Phi\leq \dimc \Theta$
avec égalité
si $\Phi\sim \Theta$. 

Soit $Z$ un sous-espace analytique fermé de $X$ contenant $x$.
Nous noterons $\Theta\cap Z$
l'ensemble $\{V\cap Z\}_{V\in \Theta}$. C'est
un filtre affinoïde de $(Z,x)$, et il résulte immédiatement 
des définitions que 
$\dimc \Theta\cap Z=\dimc \Theta$ ;  pour tout
filtre affinoïde
$\Psi$ sur $(Z,x)$ tel que $\Psi\sim\Theta\cap Z$ on a donc
$\dimc \Psi=\dimc \Theta$.

Ceci s'applique par exemple
avec $\Theta=(X\grot^\Gamma,x,\xi)$ (resp. $(X\grot^\Gamma,x)$, resp. 
$(X,x)$ si $X$ est bon) et $\Psi=
(Z^\Gamma\grot,x,\xi)$ (resp. $\mathsf (Z^\Gamma\grot,x,\xi)$, resp. 
$(Z,x)$) (nous avons utilisé
le fait que $\widetilde{(Z,x)}^\Gamma=\widetilde{(X,x)}^\Gamma$ si bien que $\xi$ peut être 
vu comme appartenant à  $\widetilde{(Z,x)}^\Gamma$). Précisons que dans chacun de ces trois 
cas la relation 
$ \Theta\cap Z\preceq \Psi$ est évidente et la relation
$\Psi\preceq \Theta\cap Z$ provient du fait que $\Psi$ est un sous-ensemble
cofinal de $\Theta\cap Z$. 

\begin{enonce}[remark]{Exemples}\label{exemples-raffinement}
Si
l'espace $X$ est $\Gamma$-strict et si $\xi$
appartient à $\widetilde{(X,x)}^\Gamma$
alors 
$(X\grot^\Gamma,x)\preceq
(X\grot^\Gamma,x,\xi)$. Si
de plus
l'espace  $X$ est bon alors $(X\grot^\Gamma,x,\xi)
\preceq (X,x)$,
et si de surcroît 
la composante homogène $\hrt x^\Gamma$
du corps résiduel \textit{gradué} $\hrt x$ est constituée d'éléments algébriques
sur le corps gradué $\widetilde k$, il résulte immédiatement
de la théorie de la réduction des germes de Temkin que pour tout 
$V$ appartenant à 
$(X^\Gamma\grot,x)$ 
alors $\widetilde{(V,x)}=\widetilde {(X,x)}$, ce qui veut dire que $V$ est un voisinage
de $x$. On a donc dans ce cas 
\[
(X^\Gamma\grot,x,\xi)
=(X^\Gamma\grot,x)
\sim (X,x).\]
\end{enonce}

\begin{exem}\label{exem-rayon}
Nous allons reprendre essentiellement ici l'exemple
détaillé à la section 4.4 de \cite{ducros2018} (suggéré par Temkin, et destiné
originellement à montrer que la platitude
naïve n'est en général pas stable par changement de base). 

Supposons donnés un réel $r>0$ et une série
entière $f=\sum a_i T^i\in k[\![t]\!]$ de rayon exactement $r$. 
On voit $f$ comme une fonction analytique sur le disque fermé $\mathscr M(k\{T/r\})$, définissant donc 
un morphisme de $\mathscr M(k\{T/r\})$ vers $\A^{1,\mathrm{an}}_k$ dont on considère le graphe 
\[\phi=(\mathrm{Id}, f)\colon M(k\{T/r\})\to M(k\{T/r\})\times_k \A^{1,\mathrm{an}}_k.\]
Le morphisme
$\phi$ induit un isomorphisme entre $ M(k\{T/r\})$ et un sous-espace analytique fermé $Z$ du domaine 
analytique $M(k\{T/r\})\times_k \A^{1,\mathrm{an}}_k$ de $\A^{2,\mathrm{an}}_k$. 
Posons $X=\A^{2,\mathrm{an}}_k$ et $V=M(k\{T/r\})\times_k \A^{1,\mathrm{an}}_k$, et notons $x$
le point $\phi(\eta_r)$ de $Z$.
Le morphisme $\phi$ permet d'identifier
le corps résiduel gradué $\hrt x$ à 
$\hrt \eta_r=\widetilde k(t)$ (où $t=\widetilde{T(\eta_r)}$ 
est une indéterminée de degré $r$).
Nous noterons $\xi$ le point de $\widetilde{(X,x)}$ correspondant à la valuation 
«à l'infini» sur $\hrt x=\widetilde  k(t)$ (qui envoie une fraction sur l'opposé de son degré monomial).
On a évidemment $\dim_x X=2$ et $d_k(x)=1$. 
On démontre (\cite{ducros2018}, preuve 
de la proposition 4.4.6) que 
$\dimc (X,x)=2$. Autrement dit, $x$ est Zariski-dense dans chacun de ses
voisinages analytiques connexes
dans $X$ (de tels voisinages sont irréductibles par normalité de l'espace affine). 
En revanche, $\adhp xV$ coïncide
avec la courbe $Z$, qui est de dimension $1$. 
Compte-tenu du fait que 
$d_k(x)=1$ (ce qui oblige toute dimension centrale en $x$, 
quel que soit le filtre affinoïde considéré,  
à valoir au moins 1), il vient
 \[\dimc  (X\grot,x)=\dimc(X\grot^{r^\Z},x)=\dimc (V,x)=1.\]
 Mentionnons
enfin pour information que $\dimc (X\grot,x,\xi)=2$ ; nous 
ne le démontrerons pas ici, mais on peut le voir en adaptant la preuve de la proposition 4.4.6
de \cite{ducros2018} (qui consiste précisément à s'assurer que la courbe $Z$ ne peut pas être prolongée
dans la direction définie par $\xi$). 
Supposons maintenant que $\abs {k^\times}\neq \{1\}$
et que $r$ n'appartient pas à $\abs{k^\times}^\Q$. 
On a alors $\hrt{\eta_r}^1=\widetilde k^1$, 
ce qui entraîne
que $(X\grot^{\{1\}},x)
\sim (X,x)$ (exemples \ref{exemples-raffinement})
et partant que 
$\dimc (X\grot^{\{1\}},x)=
\dimc (X,x)=2.$
\end{exem}

Soit $X$ un bon espace $k$-analytique
et soit $x\in X$. L'anneau local $\mathscr O_{X,x}$ est
noethérien et hensélien \cite[th. 2.1.4 et 2.1.5]{berkovich1993},  
et même excellent \cite[th. 2.13]{ducros2009}.  
Et si $V$ est un bon domaine analytique de $X$
contenant $x$, le morphisme local $\mathscr O_{X,x}\to \mathscr O_{V,x}$
est régulier 
\cite[th. 3.3]{ducros2009} ; sa platitude est due essentiellement à Berkovich, \textit{cf.} \cite[prop. 2.2.4]{berkovich1990}, 
qui se fonde lui-même sur le cas strict, 
lequel est établi dans \cite[\S 7.3.2; cor. 6]{bosch-g-r1984}. 
On a par ailleurs d'après
\cite[cor. 3.2.9]{ducros2018} l'égalité 
\[\dimc (X,x)+\dim \mathscr O_{X,x}=\dim_x X.\] 
La proposition suivante généralise les résultats ci-dessus au cas des anneaux locaux associés à un filtre affinoïde ; précisons qu'on les utilise de manière essentielle dans la preuve (nous ne les redémontrons donc pas).

\begin{prop}\label{prop-otheta}
Soit $(X,x)$ un espace $k$-analytique
pointé et soit $\Theta$
un filtre affinoïde
sur $(X,x)$. 

\begin{enumerate}[1]
\item L'anneau $\mathscr O_\Theta$ est local, hensélien, excellent, et 
\[\dimc \Theta+\dim \mathscr O_\Theta=\dim_x X.\]
\item Soit $\Phi$ un filtre
affinoïde sur $(X,x)$
tel que $\Theta\preceq \Phi$. 
Le morphisme canonique  $\mathscr O_\Phi\to \mathscr O_\Theta$ est local 
et régulier. 
\item Pour tout $V\in \Theta$ le morphisme canonique $\mathscr O_{V,x}
\to \mathscr O_\Theta$ est local et régulier. 
\end{enumerate}

\end{prop}

\begin{proof}
Remarquons pour commencer que (3) est un cas particulier de (2), que nous avons choisi d'expliciter
car il nous semble important. En effet, il suffit pour démontrer (3) d'appliquer (2) en prenant pour $\Phi$
l'ensemble des voisinages affinoïdes de $x$ dans $V$. Il reste donc à démontrer (1) et (2).

Posons $d=\dimc \Theta$ ; 
soit $\Theta'$ le sous-ensemble
de $\Theta$ formé des domaines
$V$ tels que $\dim \adhp xV=d$. 
On a 
\[\mathscr O(\Theta)=\colim_{V\in \Theta}\mathscr O_{V,x}
=\colim_{V\in \Theta'}\mathscr O_{V,x},\]
la dernière égalité résultant du fait que $\Theta'$ est cofinal dans $\Theta$. 
Pour tout $V\in \Theta'$ nous noterons $\mathfrak m_{V,x}$
l'idéal maximal de $\mathscr O_{V,x}$, et $\kappa_V(x)$ son corps résiduel. 
Soient $V$ et $W$ deux éléments de $\Theta'$ tels que $W\subset V$. 
Il résulte de la définition de $\Theta'$ que $\dimc (V,x)=\dimc (W,x)=d$.
On déduit alors du corollaire 3.2.9 de \cite{ducros2018} que
\[\dim \mathscr O_{W,x}=\dim\mathscr O_{V,x}=\dim_x X-d\] (on utilise le fait que
$\dim_x W=\dim_x V=\dim_x X$). En particulier,
les anneaux locaux $\mathscr O_{V,x}$ 
et $\mathscr O_{W,x}$ ont même dimension.
Il s'ensuit par platitude de
la flèche $\mathscr O_{V,x}\to \mathscr O_{W,x}$
que $\mathscr O_{W,x}/\mathfrak m_{V,x}\mathscr O_{W,x}$ 
est local artinien, puis par régularité de cette même flèche que
$\mathscr O_{W,x}/\mathfrak m_{V,x}\mathscr O_{W,x}$  est une extension séparable
de $\kappa_V(x)$. Autrement dit, $\mathfrak m_{W,x}$
est égal à $\mathfrak m_{V,x}
\mathscr O_{W,x}$ et $\kappa_W(x)$ est séparable sur $\kappa_V(x)$.

Pour tout couple $(V,W)$ d'objets de $\Theta'$ avec $W\subset V$ 
le morphisme $\mathscr O_{V,x}\to \mathscr O_{W,x}$ d'anneaux locaux
noethériens est plat, et l'on vient 
de voir que $\mathfrak m_{W,x}
=\mathfrak m_{V,x}\mathscr O_{V,x}$. On déduit alors du lemme 10.3.13
du chapitre 0 de \cite{ega31} que $\mathscr O_\Theta$ est local noethérien, et que
pour tout $V\in \Theta'$ le morphisme $\mathscr O_{V,x}\to \mathscr O_\Theta$ est plat
et vérifie l'égalité $\mathfrak m_\Theta=\mathfrak m_{V,x}\mathscr O_\Theta$, où $\mathfrak m_\Theta$
désigne l'idéal maximal de $\mathscr O_\Theta$ ; ceci implique que $\dim \mathscr O_\Theta=\dim \mathscr O_{V,x}=
\dim_x X-d$. 

Enfin, comme chacun des $\mathscr O_{V,x}$ est hensélien, un argument immédiat de passage à la limite
montre que $\mathscr O_\Theta$ est hensélien.

Pour tout couple $(V,W)$ d'objets de $\Theta'$ avec $W\subset V$,
le morphisme $\mathscr O_{V,x}\to \mathscr O_{W,x}$ est plat, et 
on a	 vu que $\kappa_W(x)$ est séparable
sur $\kappa_V(x)$.
La colimite $\mathscr O_\Theta$ des $\mathscr O_{V,x}$ (pour $V\in \Theta'$) étant 
noethérienne par ce qui précède,  et chacun des $\mathscr O_{V,x}$ étant excellent, 
il résulte de \cite[cor. 4.4]{marot1979}
que
$\mathscr O_\Theta$ est lui-même excellent, ce qui achève de montrer (1).

Montrons maintenant (2).
Soit $I$ l'ensemble des couples $(U,V)$ où $U$ appartient à $\Phi$, où $V$ appartient à $\Theta$ et 
où $V\subset U$. L'ensemble $I$ est filtrant (pour l'inclusion composante par composante) : 
si $(U,V)$ et $(U',V')$ sont deux éléments de $I$, on choisit un élément $U''$ de $\Phi$ contenu dans 
$U\cap U'$, un élément $V''$ de $\Theta$ contenu dans $U''$, puis un élément $V'''$
de $\Theta$ contenu dans $V\cap V'\cap V''$, et $(U'',V''')$ est un élément de $I$ tel que
$U''\subset U\cap U'$ et $V'''\subset V\cap V'$. 
De plus tout élément $U$ de $\Phi$ est le premier terme d'un couple appartenant 
à $I$, et les éléments de $\Theta$ qui sont le second terme d'un couple appartenant à $I$
forment une partie cofinale de $\Theta$ : si $V\in \Theta$ on choisit d'abord
$U\in \Phi$, puis un élément $V'$ de $\Theta$ contenu dans $U$, et un élément $V''$
de $\Theta$ contenu dans $V\cap V'$ : on a alors $V''\subset V$ et $(V'',U)\in \Phi$. 
Si l'on décide de noter $U_i$ et $V_i$ les deux composantes d'un élément $i$ de $I$, on voit par ce qui précède
que
\[\mathscr O_\Phi=\colim_i \mathscr O_{U_i,x}\;\text{et}\;
\mathscr O_\Theta=\colim_i \mathscr O_{V_i,x}.\]
De surcroît l'on dispose pour tout $i$ d'un morphisme local
régulier de
$\mathscr O_{U_i,x}$ vers $\mathscr O_{V_i,x}$, et le morphisme
$\mathscr O_\Phi\to \mathscr O_\Theta$ est induit par la collection des
$\mathscr O_{U_i,x}\to \mathscr O_{V_i,x}$. Il résulte alors de la proposition 
1.1 de \cite{marot1979} que $\mathscr O_\Phi\to \mathscr O_\Theta$ 
est régulier.
\end{proof}

\begin{coro}\label{coro-huber-excellent}
Supposons que la valeur absolue de $k$ n'est pas triviale.  Tout anneau local d'un espace adique de type fini
sur $(k,k^\circ)$ est excellent. 
\end{coro}

\begin{proof}
C'est une conséquence directe de l'assertion (1) de la proposition 
\ref{prop-otheta} ci-dessus et de l'exemple \ref{exemple-huber}.
\end{proof}

\begin{lemm}\label{lem-strongly-finite}
Soit $(X,x)\to (S,s)$ un morphisme fini entre espaces $k$-analytiques
pointés. Soit $\Theta$ un filtre affinoïde sur $(X,x)$
et soit $\Phi$ un filtre affinoïde sur $(S,s)$ tels que $X\times_S U$ appartienne
à $\Theta$ pour tout $U$ appartenant à $\Phi$. 
On suppose que $\dimc \Theta =\dim \adhp xX$. On a alors
l'égalité
$\dimc \Phi=\dim \adhp sS$.

\end{lemm}

\begin{proof}
Quitte à remplacer
$X$ et $S$ par les
adhérences de
Zariski (disons réduites) respective
de $x$ et $s$, 
et $\Theta$ (resp. $\Phi$) par $\Theta\cap \adhp xX$
(resp. $\Phi\cap \adhp sS$)
on peut supposer que $X$ et $S$
sont irréductibles de même dimension, que
nous noterons
$d$, et que $x$ et $s$
sont Zariski-denses dans $X$ et
$S$ respectivement ; on a donc 
$\dimc \Theta=d$.
Soit $U$
un élément de $\Phi$ 
et soit $V$ le produit fibré $X\times_S U$. 
Nos hypothèse assurent tout d'abord que 
$V$ appartient à $\Theta$,  puis
que $\dim \adhp xV=d$.  
Comme $V\to U$ est fini, cette dimension
est égale à celle de $\adhp sU$,
et l'on a donc bien 
$\dimc \Phi=d$. 
\end{proof}

Nous pouvons maintenant énoncer le théorème principal de cette section, celui sur lequel
reposerons nos nouveaux critères de platitude. 

\begin{theo}\label{tres-generique-robuste}
Soit $Y\to X$ un morphisme d'espaces $k$-analytiques, soit $x$ un point de $X$
et soit $y$ un point de $Y_x$. 

\begin{enumerate}[1]
\item On suppose que 
$X$ et $Y$ sont bons, 
que 
$y$ appartient à $\mathrm{Int}(Y/X)$, que $\dimc (X,x)=\dim_x X$ et que 
$\dimc (Y_x,y)=\dim_y Y_x$. 
On a alors 
\[\dimc (Y,y)=\dim_y Y=\dim_x X+\dim_y Y_x.\]

\item On suppose que 
$X$ et $Y$ sont $\Gamma$-stricts, que 
$\dimc (X\grot^\Gamma,x)=\dim_x X$
et que $\dimc ((Y_x)\grot^\Gamma,y)=\dim_y Y_x$. 
On a alors 
\[\dimc (Y\grot^\Gamma,y)=\dim_y Y=\dim_x X+\dim_y Y_x.\]

\end{enumerate} 
\end{theo}

\begin{enonce}[remark]{Commentaires}
C'est uniquement l'assertion (2) dont nous aurons besoin par la suite, mais nous avons choisi d'inclure
l'assertion (1) dans ce théorème parce qu'elle nous a semblé
intéressante en elle-même et parce que sa démonstration est courte. La démonstration de (2) est plus délicate
et fait intervenir de manière essentielle le théorème \ref{theo-casqe}. 
\end{enonce}

\begin{proof}[Démonstration du théorème \ref{tres-generique-robuste}]
Posons $n=\dim_x X$ et $d=\dim_y Y_x$. On a alors $\dim_y Y\leq n+d$. 
Il suffit donc dans le cas (1) (resp. (2)) de démontrer que pour tout voisinage analytique 
$\Omega$ de $y$ dans $Y$ (resp. pour tout domaine analytique $\Gamma$-strict $\Omega$
de $Y$ contenant $y$) la dimension de
$\adhp y\Omega$ est égale à $n+d$. 
Quitte à remplacer $Y$ par $\Omega$, 
on voit qu'il suffit de démontrer que $\dim \adhp y Y=n+d$ (et l'on peut de surcroît
le faire après avoir restreint $X$ et $Y$ autant que de besoin). C'est ce que nous allons
faire dans chacune des situations (1) et (2). 
Remarquons d'emblée
que dans chacune des deux situations considérées
on a
l'égalité $\dim \adhp xX=\dim_x X=n$, si bien que le point $x$ est 
situé sur une unique
composante irréductible de $X$ ; on peut donc quitte à restreindre $X$ le supposer
purement de dimension $n$. 

Montrons (1). 
On peut supposer que $Y$ et $X$ sont affinoïdes, puis, en vertu de \cite[cor. 4.7]{ducros2007}, 
qu'il existe 
une factorisation
du morphisme $Y\to X$ sous la forme $Y\to T\to X$ où $T$ est affinoïde et quasi-lisse purement 
de dimension relative $d$ sur $X$ et où $Y\to T$ est fini. 
Le point $y$ appartient à  $\mathrm{Int}(Y/X)$, et son image $t$ sur $T$ 
appartient alors à $\mathrm{Int}(T/X)$. Puisque $Y\to T$ est fini, l'image de $\adhp yY$ sur $T$ est un fermé 
de Zariski $Z$ de $X$, et puisque $\dimc (Y_x,y)=d$, le fermé $Z_x$ est de dimension au moins $d$ en $t$ ; la fibre
$T_x$ étant quasi-lisse (et 
\textit{a fortiori}
normale) purement de dimension $d$, le fermé $Z_x$ contient la composante
connexe de $t$ dans $T_x$, et est en particulier un voisinage de $t$ dans $T_x$. L'égalité 
$\dimc (X,x)=\dim_x X$ entraîne que $\mathscr O_{X,x}$ est artinien
(c'est un cas particulier de la proposition \ref{prop-otheta}, mais qui est en fait une 
conséquence directe du corollaire 3.2.9 de \cite{ducros2018}), et $T\to X$ est lisse en $t$ car
$t\in \mathrm{Int}(T/X)$. La proposition 6.3.1 de \cite{ducros2018} assure
alors que $Z$ est un voisinage de $t$ dans $T$. Or comme $T$ est quasi-lisse purement de dimension relative $d$
sur $X$, il est purement de dimension $n+d$. En conséquence $\dim Z=n+d$ et $\dim \adhp yY=n+d$, ce qui termine
la démonstration de (1).

Montrons maintenant (2). 
Supposons tout d'abord que $Y\to X$ est quasi-étale (dans ce cas, l'hypothèse que 
$\dimc  ((Y_x)\grot^\Gamma,y)=\dim_y Y_x$ est vide) ; on peut par ailleurs
supposer $X$ et $Y$ compacts. 
On a alors $d=0$ 
et il s'agit de montrer que
$\adhp yY$ est de dimension $n$ ; notons $Z$
le fermé $\adhp yY$ et $T$ son image sur $X$ ; posons $m=\dim Z$. 
Le théorème \ref{theo-casqe} assure l'existence d'une filtration décroissante
$(X_d)$ 
de $X$ par des domaines
analytiques fermés et $\Gamma$-stricts
telle que $T\cap (X_d\setminus X_{d+1})$ soit
pour tout $d$ un fermé de Zariski de $X_d\setminus X_{d+1}$ purement de dimension $m$, 
et telle que $X_d$ soit vide pour $d$ assez grand. 
Soit $d$ l'unique entier tel que $x\in X_d\setminus X_{d+1}$. 
Le fermé de Zariski $T\cap (X_d\setminus X_{d+1})$ de  $X_d\setminus X_{d+1}$
contenant le point $x$, il est de dimension $n$ en vertu de l'égalité
$\dimc (X\grot^\Gamma,x)=n$ ; puisqu'il est par ailleurs
purement de dimension $m$, il vient $m=n$, ce qui achève la preuve de (2) 
dans le cas quasi-étale.

Pour montrer (2) dans le cas général, on peut supposer que $X$ et $Y$ sont affinoïdes
puis, quitte à restreindre $Y$, 
que $Y\to X$ admet une factorisation 
$Y\to T\to X$ où $T$ est affinoïde et $\Gamma$-strict, où $T\to X$ est quasi-lisse purement
de dimension relative $d$, et où $Y\to T$ est fini 
(\cite[cor. 4.7]{ducros2007} ; qu'on puisse de surcroît supposer que $T$ est $\Gamma$-strict
est expliqué au début de la preuve du théorème 8.2.5 de \cite{ducros2018}). 
L'image $t$ de $y$ sur $T_x$ satisfait 
l'égalité $\dimc ((T_x)\grot^\Gamma)=\dim_t T_x$
(lemme \ref{lem-strongly-finite}) et il suffit pour 
montrer que $\dim \adhp yY=n+d$ de s'assurer que
$\dim \adhp tT=n+d$. Cela permet de
supposer que $Y\to X$ est quasi-lisse purement de dimension relative $d$, et que $Y$
et $X$ sont  compacts (ils sont même affinoïdes, 
mais nous n'aurons plus besoin de cette dernière propriété qui 
pourrait ne pas être préservée lors de réductions ultérieures ; la compacité sera par contre essentielle). 
D'après le théorème 3.11 de \cite{ducros2021a}, il existe une famille finie de morphismes quasi-étales $X_i\to X$ à sources affinoïdes 
et $\Gamma$-strictes
dont les images recouvrent $X$ et, pour tout $i$, un recouvrement fini $(Y_{ij})_j$ de $Y\times_X X_i$ par des domaines
analytiques compacts et $\Gamma$-stricts 
tels que les fibres de $Y_{ij}\to X_i$ soient géométriquement connexes. 
Il existe un couple d'indices $(i,j)$ et un antécédent $\eta$ de $y$ sur $Y_{ij}$, dont l'image $\xi$
sur $X_i$ est un antécédent de $x$, et satisfait donc
d'après le cas quasi-étale déjà traité ci-dessus
l'égalité $\dimc((X_i)\grot^\Gamma,\xi)=\dim_{\xi}X_i=n$ ; pour la même raison, 
l'antécédent $\eta$ de $y$ relativement au morphisme
quasi-étale $Y_{ij,\xi}\to Y_x$
vérifie l'égalité
$\dimc ((Y_{ij,\xi})\grot^\Gamma,\eta)=
\dim_\eta Y_{ij,\xi}$. Il suffit pour conclure de démontrer que
$\adhp \eta{(Y_{ij})}$ est de dimension $d+n$ ; autrement dit, on s'est ramené au cas où $Y\to X$ est à fibres
géométriquement connexes. Comme ce morphisme est quasi-lisse, il est plat et son image est donc un domaine
analytique compact et
$\Gamma$-strict de $X$ ; en remplaçant $X$ par ce dernier, on se ramène au cas où toutes les fibres de $Y\to X$ sont non vides, 
et partant géométriquement irréductibles.
Soit $Z$ l'adhérence de Zariski de $y$ dans $Y$, munie de sa structure réduite. Puisque $Y\to X$
est quasi-lisse à fibres géométriquement irréductibles, le faisceau
d'idéaux définissant $Z$ possède un «idéal des coefficients» définissant
un sous-espace analytique fermé $T$ de $X$ \cite[th. 3.9]{ducros2021a}. 
Puisque $\dimc((Y_x)\grot^\Gamma,y)=\dim_y Y_x$ 
et puisque $Y_x$ est irréductible, l'adhérence de Zariski de $y$ dans $Y_x$ est égale à $Y_x$ tout entier. 
Par conséquent, $Z_x$ coïncide ensemblistement avec $Y_x$, et est donc égal à $Y_x$ comme
espace $\hr x$-analytique puisque $Y_x$ est réduit ; on en déduit que $x$ appartient à $T$, 
qui contient dès lors (ensemblistement) la composante irréductible $X_0$ de $X$ sur laquelle
est située $x$,
laquelle est de dimension $n$. Il s'ensuit que $Z$ contient ensemblistement $Y\times_X X_0$, qui est de dimension $n+d$. 
Par conséquent, $\dim Z=n+d$. 
\end{proof}

\section{Platitude et anneaux G-locaux}
\label{section-apgen}

Nous allons maintenant exhiber de nouveaux exemples de situations
dans lesquelles la platitude se détecte 
au niveau des anneaux locaux analytiques, voire au niveau
des anneaux locaux algébriques lorsque cela peut avoir un sens, c'est-à-dire lorsque les espaces en jeu sont des analytifiés
de schémas de type fini sur une algèbre affinoïde.   
Nous renvoyons à la section 2.6 de \cite{berkovich1993} pour la définition et les propriétés de base
des analytifiés (on pourra trouver quelques compléments au chapitre 2 de \cite{ducros2018}, et plus
spécialement à la section 2.7). 

\begin{enonce}[remark]{Notations}
\label{notation-Fx}
Soit $X$ un espace $k$-analytique, soit $x$ un point de $X$ 
et soit $\mathscr F$ un faisceau cohérent sur $X$. 
Si $X$ est bon, nous noterons $\mathscr F_x$ la fibre en $x$ du faisceau $\mathscr F$ \textit{restreint
à la catégories des ouverts de $X$}. Si $X$ et $\Gamma$-strict et si 
$\xi$ désigne un point de $\widetilde{(X,x)}^\Gamma$, nous noterons 
$\mathscr F_{X\grot^\Gamma,x}$ (resp. $\mathscr F_{X\grot^\Gamma,x,\xi}$) la 
colimite des $\mathscr F(V)$ où $V$ parcourt l'ensemble des domaines affinoïdes
$\Gamma$-stricts de $X$ contenant $x$ (resp. contenant $x$ et tels que $\xi\in \widetilde{(V,x)}^\Gamma$). 

Si $\mathscr X$ est un schéma de type fini sur une algèbre affinoïde $A$, nous
désignerons par
$\mathscr X\an$ son analytifié, et par $\mathscr F\an$ l'analytifié d'un faisceau cohérent
$\mathscr F$ sur $\mathscr X$. L'application continue naturelle $\mathscr X\an\to \mathscr X$ sera notée $x\mapsto
x\al$. 
\end{enonce}

\begin{theo}\label{theo-plat-naif}
Soit $Y\to X$ un morphisme entre bons
espaces $k$-analytiques
$\Gamma$-stricts. Soit $y$
un point de $Y$ et soit $x$ son image sur $X$.

\begin{enumerate}[1]
\item Supposons que 
\[\dimc (X\grot^\Gamma,x)=\dim \adhp xX.\]
Soit $\mathscr F$ un faisceau cohérent
sur $Y$ tel que $\mathscr F_y$ soit un $\mathscr O_{X,x}$-module plat. 
Le faisceau $\mathscr F$ est alors $X$-plat en $y$.

\item Supposons que $X=\mathscr X\an$
pour un certain schéma de type fini 
$\mathscr X$ sur une algèbre
$k$-affinoïde
$\Gamma$-stricte $A$, que $Y=\mathscr Y\an$
pour un certain schéma de type fini $\mathscr Y$
sur une algèbre $A$-affinoïde
$\Gamma$-stricte $B$, et que $Y\to X$ est induit par un $A$-morphisme
$\mathscr Y\to \mathscr X$. 
Supposons de plus que 
\[\dimc (X\grot^\Gamma,x)=\dim \overline{\{x\al\}}\an .\]
Soit $\mathscr G$ un faisceau cohérent sur $\mathscr Y$
qui est $\mathscr X$-plat en $y\al$. Le faisceau cohérent
$\mathscr G\an$
est alors $X$-plat en $y$. 
\end{enumerate}
\end{theo}

\begin{enonce}[remark]{Commentaires}\label{commentaires}
Sous les hypothèses de (2), le sous-espace
$\overline{\{x\al\}}\an$ de $X$ est irréductible 
\cite[Prop. 2.7.16]{ducros2018},  et contient
évidemment $\adhp xX$. La
condition $\dimc (X\grot^\Gamma,x)=\dim \overline{\{x\al\}}\an$ revient donc 
à demander que $\dimc (X\grot^\Gamma,x)=\dim \adhp xX$
et que 
$\adhp xX=\overline{\{x\al\}}\an$, soit encore
que $\dimc (X\grot^\Gamma,x)=\dim \adhp xX$ et que
$\adhp xX$ soit 
l'analytification d'un sous-schéma fermé
$\mathscr Z$ de $\mathscr X$ (nous dirons plus brièvement que 
$\adhp xX$ est \emph{algébrique}) : en effet, si $\adhp xX=
\mathscr Z\an$ alors $x\al$ est dense dans $\mathscr Z$
pour des raisons topologiques, si bien que $\mathscr Z=\overline{\{x\al\}}$. 
Remarquons qu'en vertu des théorèmes
GAGA \cite[appendice A]{poineau2010}, 
l'algébricité de $\adhp xX$ est
automatique dès que $\mathscr X$ est propre sur $A$, 
et en particulier si $\mathscr X=\spec A$ (auquel cas $X=\mathscr M(A)$, 
mais le résultat découle alors directement de la définition de la topologie
de Zariski, sans qu'il soit nécessaire 
d'invoquer GAGA). 

\end{enonce}

\begin{proof}[Démonstration du théorème \ref{theo-plat-naif}]
L'assertion (1) est locale sur la source
et le but, ce qui permet
de supposer $X$ et $Y$ affinoïdes ; notons $\mathscr X$ 
et $\mathscr Y$ les spectres de leurs algèbres respectives. 
Le faisceau cohérent $\mathscr F$ est donné
par un $\mathscr O(Y)$-module de type fini $M$. 
Dans le diagramme commutatif
\[
\begin{tikzcd}
\mathscr O_{Y,y}&\mathscr O_{X,x}\ar[l]\\
\mathscr O_{\mathscr Y, y\al}
\ar[u]&\mathscr O_{\mathscr X, x\al}
\ar[l]\ar[u]\end{tikzcd}\]
les flèches verticales sont fidèlement plates. Il 
s'ensuit que si $M\otimes_{\mathscr O(Y)}\mathscr O_{Y,y}$ est plat sur
$\mathscr O_{X,x}$, il l'est sur $\mathscr O_{\mathscr X, x\al}$, 
ce qui entraîne à son tour
la platitude de
$M\otimes_{\mathscr O(Y)}\mathscr O_{\mathscr Y,y\al}$
sur $\mathscr O_{\mathscr Xx\al}$. 
Il suffit donc pour démontrer (1)
de s'assurer que si 
$M\otimes_{\mathscr O(Y)}\mathscr O_{\mathscr Y,y\al}$ est 
plat sur $\mathscr O_{\mathscr X,x\al}$ alors $\mathscr F$ est 
$X$-plat en $y$. Mais c'est un cas particulier de l'assertion
(2), à savoir celui où $\mathscr G$
est le faisceau cohérent sur $\mathscr Y$ associé à $M$
(notons que comme $\mathscr X=\spec A$, l'hypothèse 
que $\dimc (X\grot^\Gamma,x)=\dim \adhp xX$ entraîne 
que $\dimc (X\grot^\Gamma,x)=\dim \overline{\{x\al\}}\an$
en vertu des commentaires faits en \ref{commentaires}, et qu'on est donc bien sous les hypothèses de l'assertion (2)). 

Il suffit donc désormais de démontrer (2).
Il résulte 
des commentaires faits en \ref{commentaires}
que $\dimc (X\grot^\Gamma,x)=\dim \adhp xX$
et que $\adhp xX$ est égal
à $\overline{\{x\al\}}\an$. 
Nous allons tout d'abord 
supposer que
$\dim \adhp xX=\dim_x X$. L'assertion (2)
est locale sur $\mathscr X$ (et $\mathscr Y$), 
ce qui permet de supposer que $\mathscr X$
est affine. 
Soit $Z$ le
lieu des points de $Y$ en lesquels $\mathscr F$
n'est pas $X$-plat. C'est un fermé de Zariski de $Y$
(d'après \cite[th. 10.7.4 (2b)]{ducros2018}
c'est même
l'analytifié d'un
fermé de Zariski de $\mathscr Y$, 
mais
nous n'en aurons pas besoin). 
Raisonnons par l'absurde et supposons donc que $y\in Z$. 
Puisque $\dim \adhp xX=\dim_x X$, 
l'anneau local $\mathscr O_{X,x}$ est artinien 
(c'est un cas particulier de la proposition \ref{prop-otheta}, mais qui est en fait une 
conséquence directe du corollaire 3.2.9 de \cite{ducros2018}),
et comme il est fidèlement plat sur $\mathscr O_{\mathscr X,x\al}$, 
ce dernier est également artinien (ce qui veut dire que
$\overline{\{x\al\}}$
est une composante irréductible de 
$\mathscr X$). 
Choisissons un faisceau cohérent d'idéaux
$\mathscr I$ sur
le schéma affine $\mathscr X$ tel que
$\mathscr I_{x\al}$ soit l'idéal
maximal de $\mathscr O_{\mathscr X,x\al}$, 
et une résolution
\begin{equation}
\tag{$\ast$}
\mathscr O_{\mathscr X}^m\to\mathscr O_{\mathscr X}^n\to \mathscr
O_{\mathscr X}\to 
\mathscr O_{\mathscr X}/\mathscr I\to 0
\end{equation}
de faisceaux cohérents sur $\mathscr X$.
Par hypothèse, $\mathscr G_{y\al}$ est plat sur $\mathscr O_{\mathscr X,x\al}$ ; 
comme $\mathscr I_{x\al}$
est nilpotent, cela équivaut
en vertu du critère local de platitude
\cite[\href{https://stacks.math.columbia.edu/tag/051C}{Tag 051C}]{stacks-project}
aux deux affirmations suivantes : 
\begin{itemize}[label=$\diamond$]
\item $\mathscr G_{y\al}/
\mathscr I_{x\al}\mathscr G_{y\al}$ est
plat sur $\mathscr O_{\mathscr X,x\al}/\mathscr I_{x\al}$,
ce qui est ici
automatique puisque ce dernier est un corps ; 
\item $\mathrm{Tor}_1^{\mathscr
O_{\mathscr X,x\al}}
(\mathscr G_{y\al}, \mathscr
O_{\mathscr X
,x\al}/\mathscr I_{x\al})=0$, 
ce qui signifie que la suite
\[(\mathscr G_{y\al})^m\to 
(\mathscr G_{y\al})^n\to \mathscr 
G_{y\al}\] déduite de $(\ast)$ est exacte.
\end{itemize}
Choisissons un polyrayon $r$
constitué d'éléments de $\Gamma$, constituant une famille libre du $\Q$-espace
vectoriel $\R_+^\times/\abs{k^\times}^\Q$, tel
que $\abs{k_r^\times}\neq \{1\}$ et tel 
que $A_r$ et $B_r$ soient strictement $k_r$-affinoïdes. 
Remarquons que comme $k_r$ est une $k$-algèbre
affinoïde $\Gamma$-stricte, tout espace strictement $k_r$-analytique
est également de manière naturelle un espace $k$-analytique 
$\Gamma$-strict. 
Soient $\eta$ et $\xi$ les images respectives
de $y$ et $x$ sur $Y_r$
et $X_r$ par la section de Shilov 
\cite[1.2.16]{ducros2018} ; puisque
$y\in Z$, le point $\eta$ est situé sur $Z_r$. 
Le morphisme $Y_r\to Y$ est plat, si bien que 
$\mathscr O_{Y_r,\eta}$ est un $\mathscr O_{Y,y}$-module
plat ; puisque
$\mathscr O_{Y,y}$
est lui-même
plat sur $\mathscr O_{\mathscr Y,y\al}$, 
la suite 
$(\mathscr G\an_\eta)^m\to (\mathscr
G\an_\eta)^n\to \mathscr G\an_\eta$
est
encore exacte. 
Il en résulte qu'il existe un voisinage
strictement $k_r$-affinoïde $V$
de $\eta$ dans $Y_r$ tel que la suite
$(\mathscr G\an_V)^m\to
(\mathscr G\an_V)^n\to \mathscr G\an_V$ soit
exacte. 
La fibre $(Z_r\cap V)_\xi$ est un fermé de Zariski non vide
de l'espace strictement $\hr \xi$-analytique
$V_\xi$ (il contient $\eta$). Il possède dès lors
un point $\hr x$-rigide $\zeta$. 
D'après le lemme 10.3.6 de \cite{ducros2018}, 
il existe un
espace strictement $k_r$-analytique $X'$, un 
morphisme quasi-étale $X'\to X_r$
et
un antécédent $\zeta'$ de $\zeta$ sur $Y':=Y_r\times_{X_r}
X'$
qui appartient à 
$\mathrm{Int}(Y'/X')$ ; 
quitte à restreindre $X'$, on peut
le supposer affinoïde.
Soit $x'$ l'image de $\zeta'$
sur $X'$. Le point $x'$ appartient à $\mathsf A(X'_x)$
puisque $X'\to X_r$ est quasi-étale et puisque $\xi$
est le point de Shilov de la fibre $(X_r)_x$, et il vérifie donc
l'égalité $\dimc((X'_x)\grot^\Gamma,x')=\dim_{x'}X'_x$. 
Comme $\dimc (X\grot^\Gamma,x)=\dim_x X$, 
on déduit du théorème \ref{tres-generique-robuste}
que $\dimc((X')\grot^\Gamma,x')=\dim_{x'}X'$, ce qui implique
que $\dimc(X',x')$
est égal à $\dim_{x'}X'$ puis que 
l'anneau local $\mathscr O_{X',x'}$ est 
artinien en vertu
de la proposition \ref{prop-otheta}, ou plus directement 
du corollaire 3.2.9 de \cite{ducros2018}.
Le sous-schéma fermé
de $\mathscr X$
défini par $\mathscr I$ contient $x\al$
et est réduit en ce dernier ; il en
résulte que le sous-espace analytique
fermé de $X$ défini par $\mathscr I\an$
contient $x$ et est réduit
en de dernier. 
Puisque 
$X'\to X$ est quasi-lisse (comme composé d'une flèche
quasi-étale 
et de $X_r\to X$),
on en déduit que le sous-espace
analytique fermé de $X'$ défini par $\mathscr I\an\mathscr
O_{X'}$ contient $x'$ et est réduit en ce dernier ; 
l'anneau
local artinien $\mathscr O_{X',x'}/\mathscr I_{x\al}
\mathscr O_{X',x'}$ 
est dès lors
un corps. En conséquence, 
$\mathscr G\an_{\zeta'}/
\mathscr I_{x\al}\mathscr G\an_{\zeta'}$
est plat sur $\mathscr 
O_{X',x'}/\mathscr I_{x\al}
\mathscr O_{X',x'}$. Et comme la suite
$(\mathscr G\an_V)^n\to
( \mathscr G\an_V)^m\to \mathscr G\an_V$
est exacte
et que $Y'\to Y$ est plat (car quasi-étale), 
la suite
\[(\mathscr G\an_{\zeta'})^n\to
(\mathscr G\an_{\zeta'})^m\to
\mathscr G\an_{\zeta'}\] 
est exacte. En vertu du critère local 
de platitude
rappelé plus haut, 
ceci entraîne que $\mathscr G\an_{\zeta'}$ est plat
sur $\mathscr O_{X',x'}$.
Le point $\zeta'$ 
appartenant à $\mathrm{Int}(Y'/X')$,
il s'ensuit que
$\mathscr G\an_{Y'}$ est $X'$-plat en 
$\zeta'$
\cite[th. 8.3.4]{ducros2018}. 
Combiné au fait que
$Y'\to X'$ se déduit de $Y\to X$ par
changement de base plat, ceci entraîne que 
$\mathscr G\an$ est $X$-plat en l'image $z$ de $\zeta'$
sur $Y$. Mais comme $\zeta$ a été choisi sur
$Z_r$, le point $z$ appartient à $Z$, ce qui
contredit la $X$-platitude
de $\mathscr G\an$  en $z$
et achève la démonstration
de (2)
lorsque $\dim \adhp xX=\dim_x X$. 

Montrons maintenant (2) dans le cas général. 
Soit 
$\mathscr J$ un faisceau cohérent d'idéaux sur
$\mathscr X$ de lieu des zéros $\overline{\{x\al\}}$ ; 
le lieu des zéros
de $\mathscr J\an$ est alors
égal à $\adhp xX$. 
Soit 
$X'$ un bon espace analytique défini sur une extension complète
de $k$ et muni d'un morphisme $X'\to X$, soit $y'$
un point de $Y':=Y\times_X X'$ situé au-dessus de $y$
et soit $x'$ son image sur $X'$. Nous allons montrer 
que $\mathscr G\an_{y'}$ est un $\mathscr O_{X',x'}$-module
plat, ce qui permettra de conclure. 
Soit $n$ un entier. En vertu du 
critère de platitude
\cite[\href{https://stacks.math.columbia.edu/tag/0523}
{Tag 0523}]{stacks-project},
il suffit
de prouver
que $\mathscr G\an_{y'}/\mathscr J_{x\al}^n\mathscr
G\an_{y'}$
est plat sur $\mathscr
O_{X',x'}/\mathscr J_{x\al}^n\mathscr O_{X',x'}$. 
Par hypothèse, $\mathscr G_{y\al}$ est plat
sur $\mathscr O_{\mathscr X,x\al}$, ce qui entraîne que 
$\mathscr G_{y\al}/
\mathscr J_{x\al}^n\mathscr G_{y\al}$
est plat sur $\mathscr
O_{\mathscr X,x\al}/\mathscr I_{x\al}^n
\mathscr O_{X\al,x\al}$. 
Soit $\mathscr Z$ le sous-schéma
fermé de $\mathscr X$
défini par le faisceau cohérent $\mathscr J^n$ ; posons
\[\mathscr T=\mathscr Y\times_
{\mathscr X}\mathscr Z,
Z=\mathscr Z\an\;\text{et}\;
T=\mathscr T\an,\]
puis 
$Z'=X'\times_X Z$ et $T'=T\times_Z Z'$. 
Par ce qui précède, $\mathscr G_{\mathscr T}$ est
$\mathscr Z$-plat
en $y\al$. Comme on a par hypothèse
$\dimc (Z\grot^\Gamma,x)
=\dim \adhp xZ=\dim_x Z$,
il résulte du cas particulier traité plus haut
que
$\mathscr G\an_T$ est $Z$-plat
en $y$. Il s'ensuit que $\mathscr G\an_{T'}$
est $Z'$-plat en $y'$, ce qui
implique que
$\mathscr G\an_{y'}/
\mathscr J_{x\al}^n\mathscr G\an_{y'}$
est plat sur $\mathscr
O_{X',x'}/\mathscr J_{x\al}^n\mathscr O_{X',x'}$.
\end{proof}

\begin{theo}\label{plat-naif-gtop}
Soit $Y\to X$ un morphisme entre espaces $k$-analytiques $\Gamma$-stricts,
soit $y$ un point de $Y$ et soit $x$ son image sur $X$. Soit $\eta$ un point
de $\widetilde{(Y,y)}^\Gamma$ et soit $\xi$ son image sur 
$\widetilde{(X,x)}^\Gamma$. Soit $\mathscr F$ un faisceau cohérent sur $Y$. Les assertions
suivantes sont équivalentes : 
\begin{enumerate}[i]
\item $\mathscr F$ est $X$-plat en $y$ ; 
\item $\mathscr F_{Y\grot^\Gamma,y}$ est plat sur $\mathscr O_{X\grot^\Gamma,x}$. 
\end{enumerate}
De plus si elles sont satisfaites alors 
$\mathscr F_{Y\grot^\Gamma,y,\eta}$ est plat sur $\mathscr O_{X\grot^\Gamma,x,\xi}$. 

\end{theo}

\begin{proof}
Supposons que 
$\mathscr F$ est $X$-plat
en $y$. 
Soit $U$ un domaine affinoïde
$\Gamma$-strict de $X$ contenant $x$.
Si $V$ est un domaine affinoïde
$\Gamma$-strict
de $Y\times_X U$ contenant $y$, alors $\mathscr F_{V,y}$
est plat sur $\mathscr O_{U,x}$ puisque $\mathscr F$ est $X$-plat en $y$. 
Par passage à la limite en faisant varier $V$, on en déduit que
$\mathscr F_{Y\grot^\Gamma,y}$ est plat sur $\mathscr O_{U,x}$. 
Par passage à la limite en faisant varier $U$, on en déduit que $\mathscr F_{Y\grot^\Gamma,y}$
est plat sur $\mathscr O_{X\grot^\Gamma,x}$.

Réciproquement, 
supposons  que
$\mathscr F_{Y\grot^\Gamma,y}$ est
plat sur $\mathscr O_{X\grot^\Gamma,x}$. 
Soit $U$ un domaine 
affinoïde $\Gamma$-strict de $X$
contenant
$x$ et tel que
$\dim \adhp xU$
soit égale à 
$\dimc (X\grot^\Gamma,x)$. 
Comme $\mathscr O_{X\grot^\Gamma,x}$
est plat sur $\mathscr O_{U,x}$
par la proposition \ref{prop-otheta} (3),
le $\mathscr O_{U,x}$-module $\mathscr F_{Y\grot^\Gamma,y}$
est plat. 
Or ce module est la colimite filtrante des $\mathscr F_{V,y}$
pour $V$ parcourant l'ensemble des domaines affinoïdes $\Gamma$-stricts
de $Y\times_X U$ contenant $y$ ; 
il en résulte
qu'il existe un tel
$V$ avec $\mathscr F_{V,y}$ plat sur
$\mathscr O_{U,x}$,
ce qui entraîne d'après le
théorème \ref{theo-plat-naif} que 
$\mathscr F_V$ est $U$-plat en $y$, 
c'est-à-dire encore que $\mathscr F$ est 
$X$-plat en $y$.

On a donc démontré que (i)
$\iff$
(ii). 
Supposons maintenant que ces
propriétés soient satisfaites. 
Dans le diagramme commutatif 
\[
\begin{tikzcd}
\mathscr O_{X\grot^\Gamma,x}\ar[r]&\mathscr O_{Y\grot^\Gamma,y}\\
\mathscr O_{X\grot^\Gamma,x,\xi}\ar[u]\ar[r]&\mathscr O_{Y\grot^\Gamma,y,\eta}
\ar[u]\end{tikzcd}
\]
les flèches verticales sont fidèlement plates 
d'après la proposition \ref{prop-otheta} (2) ; il en résulte aussitôt
au vu de (ii) que 
$\mathscr F_{Y\grot^\Gamma,y,\eta}$
est plat sur $\mathscr
O_{X\grot^\Gamma,x,\xi}$. 
\end{proof}

\begin{exem}
Reprenons les notations de l'exemple \ref{exem-rayon}, en supposant de plus
que la valeur absolue de $k$ n'est pas triviale et que $r\notin \abs{k^\times}^\Q$. 
On a vu (à la toute fin de l'étude de cet exemple) que sous
ces hypothèses on a $\dimc (X^{\{1\}}\grot,x)=2=\dim_x X$. 
Ceci entraîne que l'anneau local $\mathscr O_{X^{\{1\}}\grot,x}$ est artinien d'après la proposition 
\ref{prop-otheta} (1), et il est réduit puisque $X$ est réduit ; c'est donc un corps. On déduit alors du théorème
\ref{plat-naif-gtop} ci-dessus que pour tout espace 
\textit{strictement} $k$-analytique $Y$ muni
d'un morphisme $Y\to X$,  tout faisceau cohérent $\mathscr F$ sur $Y$, et tout antécédent $y$
de $x$ sur $Y$, le faisceau $\mathscr F$ est $X$-plat en $y$. 

Notons que l'hypothèse que $Y$ est strict est indispensable : l'immersion $Z\hookrightarrow X$ décrite au début de l'exemple
\ref{exem-rayon} n'est en effet pas plate en $x$, puisque son changement de base le long de $V\hookrightarrow X$ en fait une immersion 
fermée d'une courbe dans une surface.

\end{exem}

%\end{proof}

\bibliographystyle{smfalpha}
\bibliography{aducros}

\end{document}